\documentclass[12pt]{article}
\usepackage[pdftex, bookmarksopen=true, bookmarks=true, unicode, setpagesize]{hyperref}
\hypersetup{colorlinks=true, linkcolor=black, citecolor=black}

\usepackage{amssymb, amsmath, amsthm}

\usepackage{xcolor}

\allowdisplaybreaks
\usepackage{cite}

\newtheorem{theorem}{Theorem}[section]
\newtheorem{corollary}[theorem]{Corollary}
\newtheorem{lemma}[theorem]{Lemma}
\newtheorem{proposition}[theorem]{Proposition}

\theoremstyle{remark}

\theoremstyle{remark}
\newtheorem{example}{Example}
\theoremstyle{remark}
\newtheorem{remark}[theorem]{Remark}

\begin{document}

\begin{center}{\Large \bf Umbral calculus over a vector space}
\end{center}

{\large Abdullah Alharthi}\\ Department of Mathematics, Swansea University, Bay Campus, Swansea SA1 8EN, U.K.;
e-mail: \texttt{2252788@swansea.ac.uk}\vspace{2mm}

{\large Eugene Lytvynov}\\ Department of Mathematics, Swansea University, Bay Campus, Swansea SA1 8EN, U.K.;
e-mail: \texttt{e.lytvynov@swansea.ac.uk}\vspace{2mm}

{\small
\begin{center}
{\bf Abstract}
 \end{center}
 
\noindent Let $V$ be a vector space over $\mathbb F=\mathbb R$ or $\mathbb C$. We develop a basis-free umbral calculus over~$V$. We define the vector space of polynomials over $V$, and polynomial sequences in it. We discuss shift-invariant operators acting in polynomials over $V$. We define polynomial sequences of binomial type and Sheffer sequences over $V$. We provide equivalent characterizations of these polynomial sequences. We prove two recurrence formulas for Sheffer sequences. With each Sheffer sequence, we associate a linear operator acting in polynomials over $V$, which we call a Sheffer operator. We prove that the set of Sheffer operators is a group for the usual product of linear operators, which is isomorphic to the Riordan group of pairs of formal tensor power series in a variable from $V$. Under the assumption that $V$ is an algebra, we 
lift every Sheffer sequence over $\mathbb F$ to a Sheffer sequence over $V$. We provide examples of such  lifting. }\vspace{2mm}
 
 
  \vspace{2mm}


\section{Introduction}

\subsection{What is the umbral calculus over a vector space}

Umbral calculus  is essentially the theory of Sheffer polynomial sequences (characterized by the exponential form of their generating function) and associated  linear operators acting in polynomials. The class of Sheffer sequences includes the sequences of binomial type and  Appell sequences.  After a long period when  umbral calculus was used for purely formal calculations, the theory became rigorous in the 1970s due to the seminal 
 works of  G.-C.~Rota, S. Roman and their co-authors. Their theory is nowadays called the modern  umbral calculus, see e.g.\ the monographs \cite{Roman,C} or \cite[Chapter~4]{KRY}.  While umbral calculus has its origins in combinatorics, it also found applications in  theory of special functions, approximation theory, probability and statistics, topology and physics, see e.g.\ the survey paper \cite{DiBL} and the references therein. 
 
 An important 
 object of studies of  umbral calculus is the umbral composition, which 
equips the set of all Sheffer  sequences with a group structure. This group is isomorphic to the Riordan group of infinite lower triangular matrices, see \cite{SGWW,ShapiroBook,HeHsuShiue,WangZhang}

A significant volume of research  has been devoted to extensions of umbral calculus to the multivariate case, see e.g.\ \cite{RomanIII} and Section 4 in \cite{DiBL} for a list of references. However, this research had a significant drawback of being basis-dependent. 

Concrete examples of Sheffer  sequences over certain  infinite dimensional topological vector spaces (in particular, nuclear spaces, Hilbert spaces, (LB)-spaces) appeared in analysis, probability and mathematical physics on numerous occasions,  see e.g.\ \cite{ADKS,BK,Hida,FLO,Lytvynov2003,KSWY,Grothaus,ItoKubo,Stirlingoperators,Jansen1} and the references therein.

The paper \cite{DBLR} was a pioneering  work in which some elements of basis-independent umbral calculus over a separable Hilbert space were developed.   However, no general definition of a binomial sequence or a Sheffer sequence over a Hilbert space was given in \cite{DBLR}. See also the follow-up paper \cite{DL}.

  The paper \cite{FKLO} developed foundations of infinite-dimensional umbral calculus over the nuclear space of smooth compactly supported functions on $\mathbb R^d$. It should  be noted that, while the studies in \cite{FKLO} were basis-free, the theory still significantly used the specific structure of the underlying nuclear space.  
  
  The main aim of the present paper is to develop a basis-free umbral calculus over a generic vector space. To achieve this, only arguments from combinatorics and (to some extent) linear algebra will be used.   
  
Let us briefly outline what we mean under the  umbral calculus over a vector space. Let $V$ be a vector space over $\mathbb F=\mathbb R$ or $\mathbb C$. Let $V^*$ be the (algebraic) dual space of $V$, i.e., the vector space of linear functionals defined on $V$ with values in $\mathbb F$.  For $\omega\in V^*$ and $v\in V$, we denote by $\langle \omega,v\rangle$ the dual pairing between $\omega$ and $v$, i.e., the action of the functional $\omega$ onto $v$. 

We denote by $\mathcal P(V^*)$ the vector space of polynomials over $V$.
A function $p:V^*\to\mathbb F$ belongs to $\mathcal P(V^*)$  if and only if it can be written as $p(\omega)=\sum_{k=0}^n \langle\omega^{\otimes k},f_k\rangle$, 
 where $f_k\in V^{\odot k}$ and $\omega^{\otimes k}\in (V^{\odot k})^*$ for $\omega\in V^*$. Here, $V^{\odot k}$ is the $k$th  symmetric tensor power of $V$, and we denote $V^{\odot 0}:=\mathbb F$. 
Equivalently, a function $p:V^*\to\mathbb F$ belongs to $\mathcal P(V^*)$  if and only if there exist vectors $e_1,\dots,e_N\in V$ and a multivariate polynomial $\tilde p:\mathbb F^N\to\mathbb F$ such that $p(\omega)=\tilde p(\langle\omega,e_1\rangle,\dots,\langle \omega,e_N\rangle)$. In the case where the vector space~$V$ is finite-dimensional, we can choose 
$e_1,\dots,e_N\in V$ to be a basis in $V$.  Thus, if $V=\mathbb F^N$, we can identify $V^*$ with $\mathbb F^N$ and $\mathcal P(\mathbb F^N)$ is the usual space of multivariate polynomials in $N$ variables with coefficients from $\mathbb F$. 

A polynomial sequence in $\mathcal P(V^*)$ is defined through a sequence of linear operators $P_{kn}:V^{\odot n}\to V^{\odot k}$ with $0\le k\le n$ such that each operator $P_{nn}$ is invertible. Elements of a polynomial sequence are polynomials of the form $\sum_{k=0}^n\langle\omega^{\otimes k},P_{kn}f_n\rangle$, where $f_n\in V^{\odot m}$. In the one-dimensional case, $V=\mathbb F$, each $P_{kn}$ is just the coefficient by the $k$th power of the variable in the polynomial of degree $n$.

In umbral calculus, formal power series play a central role.  In the umbral calculus over $V$, we use formal tensor power series in variable from $V$ with values in a vector space~$W$ (typically $W=V^{\odot n}$).  Such a formal series has the form $A(v)=A_0+\sum_{k=1}^\infty A_kv^{\otimes k}$, where $A_0\in W$ and each $A_k:V^{\odot k}\to W$ is a linear operator; compare with \cite{FKLO}.

With the help of  formal tensor power series, we define the generating function $G(\omega,v)$ of a polynomial sequence over $V$. Here, for each fixed $\omega\in V^*$, $G(\omega,v)$ is a formal tensor power series in variable $v\in V$ with values in $\mathbb F$. 
We define a Sheffer sequence over $V$ as a polynomial sequence for which $G(\omega,v)=\exp\big(\langle \omega, B(v)\rangle)A(v)$, where $B(v)=\sum_{n=1}^\infty B_nv^{\otimes n}$ is a $V$-valued formal tensor power series in which the linear operator $B_1:V\to V$ is invertible, and $A(v)$ is an $\mathbb F$-valued formal tensor power series with $A(0)\ne0$. 
The umbral calculus over vector space~$V$ studies the class of Sheffer sequences over $V$ and 
associated linear operators acting in $\mathcal P(V^*)$. In line with the finite-dimensional setting, the class of Sheffer sequences over $V$ contains the subclass of sequences of binomial type (for which $A(v)$ is a non-zero constant) and Appell sequences (for which $B(v)=B_1v$).  

To develop the umbral calculus over $V$, one needs, in particular,  to study shift-invariant operators. These are linear operators acting in $\mathcal P(V^*)$ which commute with all shift operators  $E(\theta)$ ($\theta\in V^*$). Here $E(\theta)$ is a linear operator acting in $\mathcal P(V^*)$ defined by $(E(\theta)p)(\omega)=p(\omega+\theta)$ for $p\in\mathcal P(V^*)$.

Most of the known important examples of Sheffer sequences over an infinite-dim\-ens\-ional space $V$ appeared through a certain procedure of lifting of a Sheffer sequence over $\mathbb F$ to a Sheffer sequence over $V$. In this paper, we propose a general definition of such a lifting, based on the assumption that the underlying vector space~$V$ is an algebra. 
This allows us not only to cover all the existing constructions (compare with \cite[Sections~5 and~7]{FKLO}), but also to come up with new interesting examples of lifting, which will hopefully find applications in the future.

\subsection{Comparison with prior studies}
  
  We already mentioned above that in all prior studies of Sheffer sequences over an infinite dimensional space $V$, the space $V$ was assumed to be a certain locally convex topological vector space (l.c.s.). Hence, instead of using $V^{\odot n}$, the $n$th algebraic symmetric tensor power of $V$, one used a certain l.c.s.\  $V^{\widehat\odot n}$, of which $V^{\odot n}$ is a  dense subspace.  This naturally led to the assumption that the linear operators $P_{kn}:V^{\widehat \odot n}\to V^{\widehat\odot k}$ in the definition of a polynomial sequence are continuous.   While such an operator $P_{kn}$ is fully determined by its values on the dense subspace $V^{\odot n}$, the image of $V^{\odot n}$ under $P_{kn}$ is not necessarily a subset of $V^{\odot k}$. So, one might think that our theory is not as reach as it could be if we dealt with locally convex topological vector spaces $V$ and $V^{\widehat\odot n}$.  However, it  immediately follows from the definition of a Sheffer sequence that, even if we did assume that
 $V$ is a l.c.s\ and used certain topological tensor powers $V^{\widehat\odot n}$, we would still have the property that $P_{kn}V^{\odot n}\subset V^{\odot k}$.  Hence, our theory  includes all special cases of a l.c.s.\ $V$ considered in the literature.

\subsection{Organization of the paper and main results} 

The paper is organized as follows. In Section~\ref{gfxsthshtj}, we discuss preliminaries related to the symmetric tensor power of a vector space $V$,  formal tensor power series in variable from~$V$, and the Riordan group over $V$ whose elements are pairs of formal tensor power series. 

In Section~\ref{ctrdstrst6e} we define and study polynomial sequences over $V$ and shift-invariant operators. The main result of this section is the operator expansion theorem (Theorem~\ref{bgnfm758}), which provides an expansion of each shift-invariant operator through a shift-invariant lowering differential for a polynomial sequence. In the one-dimensional case, $V=\mathbb F$,  this result is known as the first expansion theorem  \cite[Section~3]{RotaKahanerOdlyzko}. 

In Section~\ref{cgfdxyrdrek}, we define a polynomial sequence over $V$ of binomial type and prove equivalent characterizations of such a polynomial sequence  (Theorem~\ref{45678765g}). In particular, we show that a polynomial sequence with $P_{0n}=0$ for all $n\in\mathbb N$ 
is of binomial type if and only if its lowering operators are shift-invariant.  Next, in Section~\ref{cxzsrarewtytyrf}, we prove equivalent characterizations of a Sheffer sequence  over $V$ (Theorem~\ref{ctrst5u}). Both Theorems~\ref{45678765g} and~\ref{ctrst5u} extend the classical results of umbral calculus (in the case $V=\mathbb F$), see e.g.\ \cite[Sections~4.3 and~4.4]{KRY}.  For a special choice of a l.c.s.\ $V$, similar characterizations were obtained in \cite{FKLO}. 
We also find an explicit formula for each Sheffer polynomial of degree $n$ (Corollary~\ref{dxzdrzdszrsxddx}). This formula uses the summation over all set partitions of $\{1,\dots,n\}$.  

In Section~\ref{cxdzrzrzarwar}, we prove two recurrence formulas for Sheffer sequences. The first formula in the case $V=\mathbb F$ is  Roman's recurrence formula for the raising operator \cite[Theorem~3.7.1]{Roman}, which in the special case of a binomial sequence was proved in  \cite[Section~4, Theorem~4~(4)]{RotaKahanerOdlyzko} and called therein the Rodrigues formula. The second formula provides a representation of a Sheffer polynomial of degree $n$, multiplied by a monomial, through Sheffer polynomials of degrees $0,1,\dots,n+1$.  The way we present both recurrence formulas in Section~\ref{cxdzrzrzarwar} is new even in the multivariate setting ($\operatorname{dim}(V)<\infty$). 

In Section~\ref{cfdrseepioioi}, we associate with each Sheffer sequence over $V$ a linear operator acting in $\mathcal P(V^*)$, which we call a Sheffer operator. We prove that the set of Sheffer operators forms a group for the usual product of linear operators, and this group is isomorphic to the Riordan group of pairs of formal tensor power series introduced in Section~\ref{gfxsthshtj}. This extends the results of papers  \cite{HeHsuShiue,WangZhang} which deal with the one-dimensional setting ($V=\mathbb F$) and a result of the paper \cite{FLO}, which deals the case where $V$ is a certain l.c.s.

Finally, in Section~\ref{gxzrfsarh}, under the assumption that the vector space $V$ is an algebra, we define and study Sheffer sequences over $V$ that are obtained from Sheffer sequences over~$\mathbb F$ through a lifting procedure. We also discuss a number of examples of such a setting, which include examples  already studied  in the literature as well as new examples.

\section{Preliminaries}\label{gfxsthshtj}

\subsection{Symmetric tensor product}

Let $\mathbb F=\mathbb R$ or $\mathbb C$, and let $V_i$ ($i=1,2$) be vector spaces over $\mathbb F$. We denote by $V_1\otimes V_2$ the (algebraic) tensor product of $V_1$ and $V_2$, see e.g.\ \cite[Chapter~14]{Roman_Algebra} or \cite[Chapter~1]{Ryan}. Recall that $V_1\otimes V_2$ is a vector space over $\mathbb F$, and it is spanned by  vectors $v_1\otimes v_2$ with $v_1\in V_1$ and $v_2\in V_2$. The tensor product is associative, i.e., for three vector spaces $V_1$, $V_2$, $V_3$, the vector spaces $(V_1\otimes V_2)\otimes V_3$ and $V_1\otimes(V_2\otimes V_3)$ can be naturally identified, thus yielding the vector space $V_1\otimes V_2\otimes V_3$.  Observe that, for a vector space $V$, both tensor products $\mathbb F\otimes V$ and $V \otimes\mathbb F$ can be naturally identified with $V$. 
For $n\ge2$ and $v\in V$, we denote by $V^{\otimes n}$ and $v^{\otimes n}$ the  $n$th tensor power of $V$ and $v$, respectively.   We will also denote $V^{\otimes 1}:=V$, $v^{\otimes 1}:=v$, $V^{\otimes 0}:=\mathbb F$, and $v^{\otimes 0}:=1$. 

The {\it symmetric tensor product of vectors $v_1,v_2,\dots,v_n\in V$} is defined by
$$v_1\odot v_2\odot\dots\odot v_n:=\frac1{n!}\sum_{\pi\in S_n}v_{\pi(1)}\otimes v_{\pi(2)}\otimes \dots \otimes v_{\pi(n)},$$
where $S_n$ is the symmetric group of order $n$.  
Note that, for each $\pi\in S_n$,
\begin{equation}\label{xeras4waq53y}
v_{\pi(1)}\odot v_{\pi(2)}\odot\dots\odot v_{\pi(n)}=v_1\odot v_2\odot\dots\odot v_n.
\end{equation}
The {\it $n$th symmetric tensor power of $V$}, denoted by $V^{\odot n}$, is defined as the subspace 
of~$V^{\otimes n}$  that is spanned by vectors $v_1\odot v_2\odot\dots\odot v_n$ with  $v_1,v_2,\dots,v_n\in V$. 

The following polarization formula holds: for any $v_1,v_2,\dots,v_n\in V$,
\[
v_1\odot v_2\odot \dots\odot v_n=\frac{1}{2^{n}n!} \sum_{ k_1 \in\{-1,1\}, \ldots,  k_n \in\{-1,1\}} k_1 k_2 \dots k_n (k_1 v_1+k_2 v_2 + \dots + k_n v_n)^{\otimes n}.
\]
 This formula immediately implies that
 \begin{equation}\label{cxtuewu533}
V^{\odot    n}=\operatorname{l.s.}\{v^{\otimes n}\mid v\in V\}.
\end{equation}
Here and below,  $\operatorname{l.s.}$ denotes the linear span.

For vector spaces $V$ and $W$, we denote by $\mathcal L(V,W)$ the vector space of all linear operators (maps) acting from $V$ into $W$. If $V=W$, we denote $\mathcal L(V):=\mathcal L(V,V)$. 

For vector spaces $V_i$, $W_i$ and linear operator $A_i\in\mathcal L(V_i,W_i)$ ($i=1,\dots,n$), one defines the tensor product $A_1\otimes A_2\otimes\dots\otimes A_n\in\mathcal L(V_1\otimes V_2\otimes\dots\otimes V_n, W_1\otimes W_2\otimes\dots\otimes W_n)$  by
\begin{equation*}
(A_1\otimes A_2\otimes\dots\otimes A_n)v_1\otimes v_2\otimes\dots\otimes v_n:=(A_1v_1)\otimes (A_2v_2)\otimes\dots\otimes(A_nv_n),\end{equation*}
where $v_i\in V_i$. For vector spaces $V$ and $W$, the  {\it symmetric tensor product of  operators $A_1,A_2,\dots,A_n\in\mathcal L(V,W)$} is    the linear operator $A_1\odot A_2\odot\dots\odot A_n\in\mathcal L(V^{\odot n},W^{\odot n})$ defined  by
\begin{equation*} A_1\odot A_2\odot\dots\odot A_n:= \frac1{n!}\sum_{\pi\in S_n}A_{\pi(1)}\otimes A_{\pi(2)}\otimes \dots \otimes A_{\pi(n)}\restriction_{V^{\odot n}}\,.\end{equation*}
Thus, for each $v\in V$,
\begin{equation*}
(A_1\odot A_2\odot\dots\odot A_n)v^{\otimes n}=(A_1v)\odot(A_2v)\odot\dots\odot(A_nv).
\end{equation*}
More generally, for linear operators $A_i\in\mathcal L(V^{\odot l_i},W)$ with $l_i\in\mathbb N$ ($i=1,\dots,n$) and $m:=l_1+l_2+\dots+l_n$, we define the linear operator $A_1\odot A_2\odot\dots\odot A_n\in\mathcal L(V^{\odot m},W^{\odot n})$ that satisfies, for all $v\in V$,
\begin{equation}\label{dxzdszsrerawr4}
(A_1\odot A_2\odot\dots\odot A_n)v^{\otimes m}
=(A_1v^{\otimes l_1})\odot(A_2v^{\otimes l_2})\odot\dots\odot(A_nv^{\otimes l_n}).
\end{equation}
Note that a formula similar to \eqref{xeras4waq53y} holds for the symmetric tensor product of linear operators. 

The  {\it dual of a vector space $V$} is defined as the vector space $V^*:=\mathcal L(V,\mathbb F)$. For $\omega\in V^*$ and $v\in V$, we  denote by $\langle\omega,v\rangle$ the dual pairing between $\omega$ and $v$. For each $\omega\in V^*$ and $n\ge 2$, we obviously have $\omega^{\otimes n}\in (V^{\odot n})^*$ and 
$$\langle\omega^{\otimes n},v_1\odot v_2\odot\dots\odot v_n\rangle=\langle\omega,v_1\rangle\langle\omega,v_2\rangle\dotsm \langle\omega,v_n\rangle,\quad v_1,v_2,\dots,v_n\in V.$$
Note, however, that if the vector space $V$ is infinite-dimensional, $(V^*)^{\odot n}$ is a proper subspace of $(V^{\odot n})^*$.

\subsection{Formal tensor power series and the Riordan group}\label{cftxsdrszresrtes}

In this section, we will extend the concept of a formal tensor power series from \cite{FKLO,FLO}   to the case of vector spaces. 

Let $V$ and $W$ be vector spaces and let $A_n\in\mathcal L(V^{\odot n},W)$ ($n\ge2$). In view of formula~\eqref{cxtuewu533}, the linear operator $A_n$ is uniquely determined by its action on vectors~$v^{\otimes n}$ with $v\in V$. Let $A_0\in W$ and $A_n\in\mathcal L(V^{\odot n},W)$ for $n\in\mathbb N$. (We note that $W$ can be identified with $\mathcal L(V^{\odot 0},W)=\mathcal L(\mathbb F,W)$.) A {\it formal tensor power series in argument from~$V$ with values in $W$} is defined  as a formal series $A(v)=A_0+\sum_{n=1}^\infty A_nv^{\otimes n}$. 

Fix $v\in V$. Then for $t\in\mathbb F$, $A(tv)=A_0+\sum_{n=1}^\infty t^n A_nv^{\otimes n}$ is a formal power series in argument $t$ with coefficients from $W$.  Thus, each formal tensor power series $A(v)$ uniquely identifies the sequence of linear operators $A_n\in\mathcal L(V^{\odot n},W)$ ($n\in\mathbb N_0$). Observe that $A_0=A(0)$. 

 We denote by $\mathcal F(V,W)$ the set of all formal tensor power series in argument from~$V$ with values in $W$. Similarly to the case of usual formal power series (e.g.\ \cite[Chapter~1, Section~2]{Roman}), $\mathcal F(V,W)$  is a vector space.

In the case  $W=\mathbb F$, we will write an element of $\mathcal F(V,\mathbb F)$ as $A(v)=\sum_{n=0}^\infty\langle A_n,v^{\otimes n}\rangle$. Here $A_n\in(V^{\odot n})^*$ and $\langle A_0,v^{\otimes 0}\rangle:=A_0\in\mathbb F$. 

For each $A(v)=A_0+\sum_{n=1}^\infty A_nv^{\otimes n}\in\mathcal F(V,W)$ and $\omega\in W^*$, we define 
$$\langle\omega,A(v)\rangle:=\langle\omega,A_0\rangle+\sum_{n=1}^\infty\langle\omega, A_nv^{\otimes n}\rangle=\langle\omega,A_0\rangle+\sum_{n=1}^\infty \langle \omega A_n,v^{\otimes n}\rangle\in\mathcal F(V,\mathbb F).$$ 

Let $A^{(i)}(v)=\sum_{n=0}^\infty\langle A_n^{(i)},v^{\otimes n}\rangle\in \mathcal F(V;\mathbb F)$ ($i=1,2$). By formula~\eqref{dxzdszsrerawr4}, 
we have, for any $m,n\in\mathbb N$, 
\begin{equation}\label{vcytdy}
\langle A_m^{(1)},v^{\otimes m}\rangle\langle A_n^{(2)},v^{\otimes n}\rangle=\langle A_m^{(1)}\odot A_n^{(2)},v^{\otimes(m+n)}\rangle.
\end{equation} 
Hence, we define the product $A(v)=A^{(1)}(v)A^{(2)}(v)=\sum_{n=0}^\infty \langle A_n,v^{\otimes n}\rangle\in \mathcal F(V,\mathbb F)$ by 
$A_n:=\sum_{k=0}^n A^{(1)}_k\odot A^{(2)}_{n-k}$\,. The $\mathcal F(V;\mathbb{F})$ is obviously an algebra under addition and (commutative) multiplication of formal tensor power series.

More generally, for vector spaces $V$ and $W$ and  formal tensor power series $A^{(1)}(v)\in\mathcal F(V,\mathbb F)$ and $A^{(2)}(v)\in\mathcal F(V,W)$, we may similarly define a formal tensor power series\linebreak $A^{(1)}(v)A^{(2)}(v)\in\mathcal F(V,W)$. 

We denote by $\mathcal F_0(V)$ the set of all $A(v)\in \mathcal F(V,\mathbb{F})$ such that $A(0)\neq 0$.  It is easy to show that, for each $A(v)\in\mathcal F_0(V)$,  there exists $A^{-1}(v)\in \mathcal F_0(V)$ that satisfies $A(v)A^{-1}(v)=1$. Hence,  $\mathcal F_0(V)$ is an abelian group for the multiplicative product of formal tensor power series; compare with \cite[Proposition~A.1]{FKLO}.

 Let $B(v)=\sum_{n=1}^\infty B_nv^{\otimes n}\in\mathcal F(V,V)$ be such that $B(0)=0$.  In view of \eqref{dxzdszsrerawr4}, we define, for 
 $n\ge2$, 
$B(v)^{\otimes n}=\sum_{m=n}^\infty C_m v^{\otimes m}\in\mathcal F(V,V^{\odot n})$ 
by 
\begin{equation}\label{frsters5tsw5u} C_m:= \sum_{i_1,\dots,i_m \geq 1 \atop i_1+i_2+\dots+i_m=n} B_{i_1}\odot B_{i_2}\odot \dots \odot B_{i_m}. \end{equation}

Let $A(v)=A_0+\sum_{n=1}^\infty A_nv^{\otimes n}\in\mathcal F(V,W)$ and let $B(v)=\sum_{n=1}^\infty B_nv^{\otimes n}\in\mathcal F(V,V)$ be such that $B(0)=0$. In view of \eqref{frsters5tsw5u}, we define the  {\it composition of $A(v)$ and $B(v)$}, or the {\it substitution of $B(v)$ into $A(v)$}, by 
\begin{equation}\label{crstese5ws5}
A(v)\circ B(v)=A(B(v)):=A_0+\sum_{n=1}^\infty A_n B(v)^{\otimes n}=A_0+\sum_{n=1}^\infty C_nv^{\otimes n}
\in\mathcal F(V;W),
\end{equation}
where
\begin{equation}\label{vcfxdtrstea}
C_n:=\sum_{m=1}^n A_m \sum_{i_1,\dots,i_m \geq 1 \atop i_1+i_2+\dots+i_m=n} \big(B_{i_1}\odot B_{i_2}\odot \dots \odot B_{i_m}\big).\end{equation}

We denote by $\mathcal F_1(V)$ the set of all formal tensor power series $B(v)=\sum_{n=1}^\infty B_nv^{\otimes n}\in\mathcal F(V,V)$ such that $B(0)=0$ and the operator $B_1\in\mathcal L(V)$ is invertible.  By using formulas \eqref{crstese5ws5} and \eqref{vcfxdtrstea}, one can easily prove that $\mathcal F_1(V)$ is a (noncommutative) group for the operation of composition of formal tensor power series; compare with \cite[Proposition~A.11]{FKLO}. Observe that the identity element in $\mathcal F_1(V)$ is $B(v)=v$. For a general $B(v)\in\mathcal F_1(V)$, we will denote by $B^{\langle-1\rangle}(v)$ the inverse element of $B(v)$ in $\mathcal F_1(V)$.

Similarly, for a formal power series $A(t)=\sum_{n=0}^\infty a_nt^n\in\mathcal F(\mathbb F,\mathbb F)$
and a formal tensor power power series $B(v)\in\mathcal F(V,\mathbb F)$ with $B(0)=0$, we define the composition 
$A(B(v)):=a_0+\sum_{n=1}^\infty a_nB(v)^n\in\mathcal F(V,\mathbb F). $ 

We finish this section with a discussion of the semidirect product of the groups $\mathcal F_0(V)$ and $\mathcal F_1(V)$, compare with \cite[Section~3.4]{FLO}. We denote $\mathcal R(V):=\mathcal F_{0}(V)\times \mathcal F_{1}(V)$ and define a product $\ast$ in $\mathcal R(V)$ as follows: for any $(A^{(1)}(v),B^{(1)}(v)),(A^{(2)}(v),B^{(2)}(v))\in \mathcal R(V)$,
\begin{equation}\label{bvfrsxrea4qw}(A^{(1)}(v),B^{(1)}(v))\ast(A^{(2)}(v),B^{(2)}(v)):=\big(A^{(1)}(B^{(2)}(v))A^{(2)}(v), B^{(1)}( B^{(2)}(v))\big).\end{equation}
It is straightforward to check that $(\mathcal R(V),\ast)$ is a group. We will call $\mathcal R(V)$ the {\it Riordan group over the vector space $V$}; compare with \cite[Section~7.3]{ShapiroBook}.

By identifying $A(v)\in\mathcal F_0(V)$ and $B(v)\in\mathcal F_1(V)$ with $(A(v), v)$ and $(1, B(v))$, respectively, we easily see that
 $\mathcal F_0(V)$ is an (abelian) normal subgroup of $\mathcal R(V)$, and $\mathcal F_1(V)$ is a subgroup of $\mathcal R(V)$. Furthermore, each $(A(v), B(v))\in \mathcal R(V)$ admits a unique representation as a product of elements from $\mathcal F_0(V)$ and $\mathcal F_1(V)$:
 $$(A(v), B(v))=(A(B^{\langle-1\rangle}(v)), v)\ast(1, B(v)).$$
 Hence, $\mathcal R(V)$ is the semidirect product of $\mathcal F_0(V)$ and $\mathcal F_1(V)$, i.e., 
 \begin{equation}\label{cxdsresresuy}\mathcal R(V)=\mathcal F_0(V) \rtimes\mathcal F_1(V).
 \end{equation}
 
 \subsection{Derivative of a formal tensor power series}

Let $V$ and $W$ be vector spaces and let $A_n\in\mathcal L(V^{\odot n},W)$ ($n\in\mathbb N$). Consider the function
$\tilde A_n:V\to W$ defined by $\tilde A_n(v):=A_nv^{\otimes n}$. (The function $\tilde A_n$ is, in fact, a homogeneous polynomial of degree $n$ with values in $W$.)

Let $\zeta\in V$. The  directional derivative of the function $\tilde A_n$ in direction $\zeta$ is the function $\tilde A_n'(\cdot\,;\zeta):V\to W$ that satisfies, for each $\theta\in W^*$,
$$\langle \theta, \tilde A_n'(v;\zeta)\rangle=\frac d{dt}\Big|_{t=0}\langle \theta, \tilde A_n(v+t\zeta)\rangle. $$
As easily seen, 
\begin{equation}\label{dxzra4ywy}\tilde A_n'(v;\zeta)=nA_n(v^{\otimes(n-1)}\odot \zeta).\end{equation}
Note that the map $V^{\odot (n-1)}\ni f_{n-1}\mapsto nA_n(f_{n-1}\odot \zeta)$ belongs to $\mathcal L(V^{\odot(n-1)},W)$.

Let  $A(v)=A_0+\sum_{n=1}^\infty A_nv^{\otimes n}\in\mathcal F(V;W)$. In view of \eqref{dxzra4ywy}, we define the {\it directional  derivative of the formal tensor power series $A(v)$ in direction $\zeta$} by
\begin{equation}\label{cxtsteyew46u}
A'(v;\zeta):=\sum_{n=1}^\infty nA_n(v^{\otimes(n-1)}\odot \zeta)=A_1\zeta+\sum_{n=1}^\infty (n+1)A_{n+1}(v^{\otimes n}\odot\zeta)\in\mathcal F(V,W).
\end{equation}

Let $A(v)=A_0+\sum_{n=1}^\infty A_nv^{\otimes n}\in\mathcal F(V,W)$ and $B(v)=B_0+\sum_{n=1}^\infty B_nv^{\otimes n}\in\mathcal F(V,V)$. In view of \eqref{cxtsteyew46u}, the {\it derivative of $A(v)$ in direction $B(v)$} is defined by
\begin{align} A'(v;B(v)):&= A_1B(v)+\sum_{n=1}^\infty (n+1)A_{n+1}(v^{\otimes n}\odot B(v))\notag\\
&=A_1B_0+\sum_{n=1}^\infty \sum_{m=0}^n (m+1)A_{m+1}(\mathbf 1_m\odot B_{n-m})v^{\otimes n}.\notag
\end{align}
Here and below, $\mathbf 1_m$ denotes the identity operator in $V^{\odot m}$.

\section{Polynomial sequences over a vector space}\label{ctrdstrst6e}

\subsection{Polynomials and linear operators acting in polynomials}\label{dfgh5678h}

Let $V$ be a vector space over $\mathbb F$, and let $V^*$ be the dual of $V$. A function $p:V^*\to \mathbb{F}$ is called a {\it polynomial over $V$} if
\begin{equation}\label{cfgxtsu}
p(\omega)=\sum_{k=0}^n\langle \omega^{\otimes k},f_k\rangle=f_0+\sum_{k=1}^n\langle \omega^{\otimes k},f_k\rangle,\quad \omega \in V^*.\end{equation}   
Here $f_0\in \mathbb{F}$ and $f_k\in V^{\odot  k}$ for $k=1,\dots,n$. If $f_n\neq0$,   we say that the polynomial $p$ is {\it of degree~$n$}. 
A polynomial of the form $\langle \omega^{\otimes n},f_n\rangle$  will be called a {\it monomial of degree~$n$}.
We denote by $\mathcal P  (V^*)$  the vector space of all polynomials over $V$. 
 By \eqref{cxtuewu533}, we have
 \begin{equation}\label{fxzrsareysa}\mathcal P  (V^*)=\operatorname{l.s.}\big\{f_0,\ \langle \cdot^{\otimes n},v^{\otimes n}\rangle \mid f_0\in\mathbb F,\ v\in V,\ n\in \mathbb{N}\big\}.
\end{equation}

\begin{lemma}\label{cgtxststs5u}
A function $p:V^*\to\mathbb F$ belongs to $\mathcal P(V^*)$ if and only if there exist linearly independent vectors $e_1,\dots,e_N\in V$  ($N\in\mathbb N$) and a multivariate polynomial $\tilde p:\mathbb F^N\to\mathbb F$ such that $p(\omega)=\tilde p(\langle\omega,e_1\rangle,\dots,\langle\omega,e_N\rangle)$.
\end{lemma}

\begin{proof} Let $p\in\mathcal P(V^*)$. Then, by \eqref{fxzrsareysa}, there exist $f_0\in\mathbb F$, $v_1,\dots,v_k\in V$ and $n_1,\dots,n_k\in\mathbb N$ such that 
\begin{equation}\label{cfcgdy}
p(\omega)=f_0+\sum_{i=1}^k\langle\omega^{\otimes n_i},v_i^{\otimes n_i}\rangle
=f_0+\sum_{i=1}^k\langle \omega,v_i\rangle^{n_i}. \end{equation}
Choose linearly independent vectors $e_1,\dots,e_N$ whose linear span coincides with the linear  span of the vectors $v_1,\dots,v_k$; in particular, $v_i=\sum_{j=1}^N c_{ij}e_j$ for each $i=1,\dots,k$. Then, by \eqref{cfcgdy}, $p(\omega)=\tilde p(\langle\omega,e_1\rangle,\dots,\langle\omega,e_N\rangle)$, where $\tilde p:\mathbb F^N\to\mathbb F$ is given by
$$\tilde p(x_1,\dots,x_N):=f_0+ \sum_{i=1}^k\bigg(\sum_{j=1}^N c_{ij}x_j\bigg)^{n_i}.$$
Thus, $\tilde p$ is a multivariate polynomial. 

To prove the converse statement, we only need to note that, for any $\omega\in V^*$, $e_1,\dots,e_N\in V$ and $i_1,\dots,i_N\in\mathbb N_0$,
\[\langle\omega,e_1\rangle^{i_1}\dotsm \langle\omega,e_N\rangle^{i_N}=\langle\omega^{\otimes(i_1+\dots+i_N)},e_1^{\odot i_1}\odot\dotsm\odot e_N^{\odot i_N}\rangle. \qedhere\]
\end{proof}

\begin{remark}
If the vector space $V$ is finite-dimensional, then, in Lemma~\ref{cgtxststs5u}, one may additionally assume that the vectors $e_1,\dots,e_N$ form a basis in $V$. Thus, the definition of a polynomial over $V$ generalizes the standard notion of a multivariate polynomial. 
\end{remark}

We denote by $\mathfrak F(V)$ the direct sum of the vector spaces $V^{\odot n}$ with $n\in\mathbb N_0$. Thus, $\mathfrak F(V)$ consists of all sequences $f=(f_n)_{n=0}^\infty$ such that 
$f_n\in V^{\odot n}$ ($n\in\mathbb N_0$) and for some $N\in\mathbb N_0$ (depending on $f$), we have $f_n=0$ for all $n\ge N$. 

\begin{lemma}\label{equation7867900}
We define a linear operator $I\in\mathcal L\big(\mathcal P(V^*),\mathfrak F(V)\big)$ which maps every polynomial $p\in\mathcal P(V^*)$ as in \eqref{cfgxtsu} to $(f_0,f_1,\dots,f_n,0,0,\dots)\in\mathfrak F(V)$. Then the map $I$ is bijective. 
\end{lemma}
 
 \begin{proof} The map $I$ is obviously surjective. To prove that $I$ is injective, it is sufficient to prove that the kernel of $I$ contains only the zero polynomial. To this end, it is sufficient to prove the following claim.
 
 {\bf Claim}.
{\it  Let $p(\cdot)=\sum_{k=0}^n\langle \cdot^{\otimes k}, f_k\rangle \in\mathcal P    (V^*)$ and let $f_n\neq 0,$ i.e., $p$ is a polynomial of degree $n.$ Then there exists $\omega \in V^*$ such that $p(\omega)\neq 0$.}

Indeed, for $\omega\in V^*$ and $t\in\mathbb{F}$, $p(t\omega)=\sum_{k=0}^nt^{k}\langle \omega^{\otimes k}, f_k\rangle$. Hence,
$\langle \omega^{\otimes n}, f_n\rangle=(n!)^{-1}\frac{d^n}{dt^n}\big\rvert_{t = 0}\,p(t\omega)$.
 Therefore, it is sufficient to prove that $\langle \omega^{\otimes n}, f_n\rangle\ne0$ for some $\omega\in V^*$.

 By formula \eqref{cxtuewu533}, $f_n=\sum_{k=1}^Kc_kv_k^{\otimes n}$, where $c_k\in \mathbb{F}$, $v_k\in V$, $k=1,\dots,K$, $K\in \mathbb N$.
Define $W:=\operatorname{l.s.}\{v_1,\dots,v_K\}$.  
Let $W^\dag$ be an algebraic complement of $W$ in $V,$ i.e., $W^\dag$ is a subspace of $V$ such that $W\cap W^\dag=\{0\}$ and $W+W^\dag=V$. Then, every element $v\in V$ admits a unique representation 
$v=w+w^\dag$, where $w\in W$ and $w^\dag\in W^\dag$. 
Each $\omega\in W^*$ may be extended to an element $\tilde \omega$ of $V^*$ by setting, for each $v=w+w^\dag\in V$,
$\langle \tilde \omega, v\rangle:=\langle \omega, w\rangle$. 

Since $f_n\in W^{\odot n}\subset V^{\odot n}$, it is therefore sufficient to prove the existence of $\omega\in W^*$ such that $\langle \omega^{\otimes n},f_n\rangle\ne0$. Since the vector space $W$ is finite-dimensional, the space $W^*$ is also  finite-dimensional (in fact, $\operatorname{dim}(W)=\operatorname{dim}(W)^*$), hence 
\begin{equation}\label{gtfdtrsrea4ra}
(W^{\odot n})^*=(W^*)^{\odot n}=\operatorname{l.s.}\{\omega^{\otimes n}\mid\omega\in W^* \}.
\end{equation}
 If $\langle\omega^{\otimes n},f_n\rangle=0$ for all $\omega\in W^*$, then, by \eqref{gtfdtrsrea4ra}, $\langle\theta_n,f_n\rangle=0$ for all $\theta_n\in (W^{\odot n})^*$, hence $f_n=0$. This is a contradiction. Therefore, there exists $\omega\in W^*$ such that $\langle\omega^{\otimes n},f_n\rangle\ne0$. 
\end{proof}

 For each $\theta \in V^*$, we define the  {\it shift operator} $E(\theta)\in \mathcal L( \mathcal P  (V^*))$ by $(E(\theta)p)(\omega):= p(\omega+\theta)$ for  $p\in \mathcal P  (V^*)$ and $\omega\in V^*$.  It is easy to see that, for each $n\in \mathbb{N}$ and $f_n \in V^{\odot  n}$,
\begin{equation}\label{equation786679}\big(E(\theta)\langle \cdot^{\otimes n},f_n\rangle\big)(\omega)=\sum_{k=0}^n {{n}\choose{k}}\langle \omega^{\otimes k}\odot \theta^{\otimes (n-k)},f_n\rangle.\end{equation}

For each $\theta \in V^*$, we define the {\it derivative in direction $\theta$}  as the operator $D(\theta)\in \mathcal L (\mathcal P  (V^*))$ given  by $(D(\theta)p)(\omega):=\frac{d}{dt}\bigr\rvert_{t = 0}\,p(\omega+t\theta)$. 
A straightforward calculation shows that, for any $f_n\in V^{\odot n}$,
\begin{equation}
\big(D(\theta)\langle \cdot^{\otimes n},f_n\rangle\big)(\omega)=n\langle\omega^{\otimes(n-1)}\odot\theta, f_n\rangle. \label{equation72167}
\end{equation}

Let $W$ be a vector space over $\mathbb F$. We denote by $\mathcal P(V^*,W)$ the {\it vector space of polynomials on $V^*$ with values in $W$.} This space is defined  as the linear span of functions of the form $V^*\ni\omega\mapsto p(\omega) w$, where $p\in\mathcal P(V^*)$
 and $w\in W$. Note that there exists a natural isomorphism between $\mathcal P(V^*,W)$ and the tensor product $\mathcal P(V^*)\otimes W$.

Let $k,m\in\mathbb N$ and $G_k\in (V^{\odot k})^*$. For  $v\in V$,  we define 
$$\langle G_k,v^{\otimes(k+m)}\rangle:=\langle G_k,v^{\otimes k}\rangle v^{\otimes m}\in V^{\odot m}.$$ 
Extending this definition by linearity, we define, for $f_{k+m}\in V^{\odot(k+m)}$, 
 $\langle G_k,f_{k+m}\rangle$ as an element of $V^{\odot m}$. Hence, the map 
 $$V^*\ni\omega\mapsto \langle\omega^{\otimes k},f_{k+m}\rangle\in V^{\odot m}$$
  belongs to $\mathcal P(V^*,V^{\odot m})$.

For $p\in \mathcal P  (V^*)$, we define the {\it differential of $p$} as the polynomial $Dp\in\mathcal P(V^*,V)$ that satisfies $\langle \theta,  (D p)(\omega)\rangle =(D(\theta)p)(\omega)$ for each $\theta \in V^*$. Thus, $D\in\mathcal L\big(\mathcal P(V^*), \mathcal P(V^*,V)\big)$.
It follows from \eqref{equation72167} that, for each $f_n\in V^{\odot n}$,
$$\big(D\langle\cdot^{\otimes n},f_n\rangle\big)(\omega)=n\langle\omega^{\otimes (n-1)},f_n\rangle,$$
 which can be formally written as $D\omega^{\otimes n}=n\omega^{\otimes (n-1)}$. 

More generally, for $k\in\mathbb N$, we define a linear operator $D^{k}\in\mathcal L\big(\mathcal P(V^*), \mathcal P(V^*,V^{\odot k})\big)$ that satisfies, for all $p\in\mathcal P(V^*)$ and $\theta_1,\dots,\theta_k \in V^*$, 
\begin{equation}\label{vcxsrearwa}
\langle \theta_1\odot\cdots\odot\theta_k,  (D^{k} p)(\omega)\rangle =(D(\theta_1)\dotsm D(\theta_k)p)(\omega).
\end{equation}
We similarly have, for each $f_n\in V^{\odot n}$, 
\begin{equation}\label{vdstsw6ue6i}
\big(D^{k}\langle\cdot^{\otimes n},f_n\rangle\big)(\omega)=(n)_k\langle\omega^{\otimes (n-k)},f_n\rangle.
\end{equation}
 Here $(n)_k:=n(n-1)\dotsm(n-k+1)$ is the falling factorial of degree $k$ evaluated at~$n$ (which is equal to zero for $k\ge n+1$.) Again, we can formally write $D^{k}\omega^{\otimes n}=(n)_k\,\omega^{\otimes(n-k)}$.

\begin{lemma}[Boole's formula] \label{fghj9876b}
We have, for each $\theta \in V^*$,
\begin{equation*} E(\theta)=\sum_{k=0}^{\infty}\frac{1}{k!}{D(\theta)}^k. \end{equation*}
\end{lemma}

\begin{proof}It is sufficient to prove that, for each $f_n\in V^{\odot n}$ ($n\in \mathbb{N}$), we have
$$ E(\theta)\langle \cdot^{\otimes n},f_n\rangle=\sum_{k=0}^{n}\frac{1}{k!}\,D(\theta)^k \langle \cdot^{\otimes n},f_n\rangle.$$
  But 
this immediately follows from \eqref{equation786679}, \eqref{vcxsrearwa}  and \eqref{vdstsw6ue6i}.
\end{proof}

We say that an operator $T\in \mathcal L (\mathcal P  (V^*))$ is {\it shift-invariant} if 
$TE(\theta)=E(\theta)T$ for all $\theta \in V^* $. We denote by $\mathcal {SI}(\mathcal P  (V^*))$ the set of all shift-invariant operators. Obviously,  $\mathcal{SI}(\mathcal P  (V^*))$ is an algebra under product and sum of operators from $\mathcal{SI}(\mathcal P  (V^*))$.

\begin{lemma}\label{rssw654w6u4weu6}
If $T\in \mathcal{SI}(\mathcal P  (V^*))$, then $T 1$ is a constant.
\end{lemma}

\begin{proof}For any $\omega, \theta \in V^*$, $$(T 1)(\omega)=(T E(\theta) 1)(\omega)=(E(\theta)T 1)(\omega)=(T 1)(\omega+\theta).$$
Setting $\omega=0$, we get $(T 1)(\theta)=(T 1)(0)$ for all $\theta \in V^*$. 
\end{proof}

\subsection{Polynomial sequences}

 We denote by $\mathbb P(V)$ the set of all linear operators $P\in\mathcal L(\mathcal P  (V^*))$ of the form  
\begin{equation}\label{xreaw4a4y}
\big(P\langle \cdot^{\otimes n},f_n\rangle\big)(\omega)= \sum_{k=0}^n\langle \omega^{\otimes k},P_{kn}f_n\rangle,\quad f_n\in V^{\odot n},\ n\in\mathbb N_0,
\end{equation} 
where $P_{kn}\in \mathcal L(V^{\odot   n}, V^{\odot   k})$ and each operator $P_{nn} \in \mathcal L(V^{\odot   n})$ is bijective, hence $P_{nn}^{-1}\in \mathcal L(V^{\odot   n})$.  We will identify each operator $P\in \mathbb P(V)$ with the infinite upper triangular block-matrix $[P_{kn}]_{k,n\in \mathbb {N}_0}$ in which $P_{kn}\in \mathcal L(V^{\odot n},V^{\odot k})$, for each $k\le n$ the operator $P_{kn}$ is as in formula~\eqref{xreaw4a4y}, 
  and $P_{kn}=0$ for $k>n$.  We denote by $\mathbb M(V)$ the subset of $\mathbb P(V)$ that consists of $P=[P_{kn}]_{k,n\in \mathbb N_0}\in \mathbb P(V)$ such that $P_{nn}= \mathbf 1_n$ for all $n\in\mathbb N_0$.

By formula \eqref{xreaw4a4y}, 
\begin{equation}\label{vgdyrdy6}
\big(P\langle \cdot^{\otimes n},f_n\rangle\big)(\omega) =
 \sum_{k=0}^n\langle P_{kn}^* \omega^{\otimes k},f_n\rangle =\langle P_n(\omega),f_n\rangle, \end{equation}
 where 
 \begin{equation}\label{gftdtrst5}P_n(\omega):=\sum_{k=0}^n P_{kn}^*\omega^{\otimes k}\in (V^{\odot n})^*,\quad n\in\mathbb N_0.
 \end{equation} 
 Here $P_{kn}^*\in\mathcal L((V^{\odot k})^*,(V^{\odot n})^*)$ is the dual (adjoint) operator of $P_{kn}$. 
  We observe that any sequence of the maps $V^*\ni\omega
\mapsto P_n(\omega)\in(V^{\odot n})^*$ ($n\in\mathbb N_0$) that are defined by formula \eqref{gftdtrst5}, with  $P_{kn}\in \mathcal L(V^{\odot n},V^{\odot k})$ and $P_{nn}$ being bijective, uniquely identifies an operator $P\in\mathbb P(V)$ through formula \eqref{vgdyrdy6}. We call $(P_n(\omega))_{n=0}^\infty$ the {\it polynomial sequence corresponding to the operator $P$}. Below we will effectively identify the operator  $P\in\mathbb P(V)$  with its polynomial sequence  $(P_n(\omega))_{n=0}^\infty$. If $P\in\mathbb M(V)$, then $P_{nn}^*=\mathbf 1$ for all $n\in\mathbb N_0$, and we call the corresponding polynomial sequence $(P_n(\omega))_{n=0}^\infty$ {\it monic}. 
If $P$ is the identity operator $\mathbf 1$, we obtain the {\it sequence of monomials} $(P_n(\omega)=\omega^{\otimes n})_{n=0}^\infty$.

 The {\it (exponential) generating function of a polynomial sequence
$(P^{(n)}(\omega))_{n=0}^\infty$} is defined by
$$G(\omega,v):= \sum_{n=0}^\infty \frac{1}{n!}\,\langle P_n(\omega),v^{\otimes n}\rangle. $$
In this formula, for each fixed $\omega\in V^*$, $G(\omega,\cdot)$ is a formal tensor power series from $\mathcal F_0(V)$.

 By using the definition of an operator from $\mathbb P(V)$ and formula \eqref{fxzrsareysa}, one can easily prove the following lemma.  
 
 \begin{lemma}\label{vcfcfydy}
   {\rm (i)} Let $P\in\mathbb P(V)$ and let $(P_n(\omega))_{n=0}^\infty$ be the corresponding polynomial sequence. Then $P$ is a bijection, and   $P^{-1}\in\mathbb P(V)$. Furthermore,
    $$\mathcal P  (V^*)=\operatorname{l.s.}\big\{f_0,\ \langle P_n(\cdot),v^{\otimes n}\rangle \mid f_0\in\mathbb F,\ v\in V,\ n\in \mathbb{N}\big\}.
$$

{\rm (ii)} The $\mathbb P(V)$ is a group for the product of linear operators. The $\mathbb M(V)$ is a subgroup of $\mathbb P(V)$.
 \end{lemma}

 Let $P\in \mathbb P(V)$. For each $\theta\in V^*$, we define the {\it lowering operator $Q(\theta)\in \mathcal L(\mathcal P  (V^*))$ corresponding to $P$} by $Q(\theta ):=PD(\theta)P^{-1}$. It follows from \eqref{equation72167} that, for $n\in\mathbb N_0$ and $v\in V$,
  \begin{align}
  \big(Q(\theta)\langle P_n(\cdot),v^{\otimes n}\rangle\big)(\omega)&=n\langle \theta,v\rangle \langle P_{n-1}( \omega),v^{\otimes (n-1)}\rangle\notag\\
  & =n\langle P_{n-1}(\omega)\odot\theta,v^{\otimes n}\rangle,\quad v\in V.\notag\end{align} 

Let $k\in\mathbb N$. Similarly to \eqref{vcxsrearwa},  we define  a linear operator $Q^{k}\in\mathcal L\big(\mathcal P(V^*), \mathcal P(V^*,V^{\odot k})\big)$ that satisfies, for all $p\in\mathcal P(V^*)$ and $\theta_1,\dots,\theta_k \in V^*$,
\begin{equation}\label{cxrse5w354w5}
\big\langle \theta_1\odot\cdots\odot\theta_k,  (Q^{k} p)(\omega)\big\rangle =(Q(\theta_1)\dotsm Q(\theta_k)p)(\omega).
\end{equation} 
We have, by \eqref{vdstsw6ue6i},
\begin{equation}\label{ge4ew34wq43w}
\big(Q^{k}\langle P_n(\cdot),f_n\rangle\big)(\omega)=(n)_k\langle P_{n-k}(\omega),f_n\rangle,
\end{equation}
which can be formally written  as $Q^{k}P_n(\omega)=(n)_kP_{n-k}(\omega)$. In the case $k=1$, we will also write $Q:=Q^{1}$ and call $Q$ the {\it lowering differential}.

\begin{proposition}(Polynomial expansion)\label{dyrde6e6e645} 
Let $P\in\mathbb P(V)$ and let $(P_n(\omega))_{n=0}^\infty$ be the corresponding polynomial sequence. Assume that  $P_{0}=1$ and $P_n(0)=0$ for all $n\in\mathbb N$. Let $Q$ be the corresponding lowering differential. Then, for each $p\in \mathcal P  (V^*)$, 
\begin{equation}\label{vcyrte65}
p(\omega)=p(0)+\sum_{k=1}^\infty \frac{1}{k!}\,\big\langle P_k(\omega),(Q^{k}p)(0)\big\rangle.
\end{equation}
\end{proposition}

\begin{proof} Formula \eqref{vcyrte65} trivially holds when $p(\omega)$ is constant. 
Since both the left- and right-hand sides of formula \eqref{vcyrte65} depend linearly on $p\in \mathcal P  (V^*)$,   it is sufficient to prove formula \eqref{vcyrte65} in the case $p(\omega)=\langle P_n(\omega), f_n \rangle$, where  $n \in \mathbb{N}$ and $f_n\in V^{\odot n}$. 
Since $P_n(0)=0$ for all $n\in\mathbb N$,  
we have, by \eqref{ge4ew34wq43w},
\begin{align*}&\langle P_n(0),f_n\rangle+ \sum_{k=1}^\infty \frac{1}{k!}\,\big\langle P_k(\omega), (Q^{k}\langle P_n(\cdot),f_n\rangle)(0)\big\rangle\\
&\quad=\sum_{k=1}^n \frac{1}{k!}\,\big\langle P_k(\omega),
(n)_k\langle P_{n-k}(0),f_n\rangle\big\rangle\\
&\quad=\frac1{n!}\langle P_n(\omega), n!\, P_0f_n
\rangle=\langle P_n(\omega),f_n\rangle. \qedhere
\end{align*} 
\end{proof}

\begin{remark} In the special case $P=\mathbf 1\in\mathbb P(V)$, formula \eqref{vcyrte65} can be written  as follows:
$p(\omega)=p(0)+\sum_{k=1}^\infty\frac1{k!}(D(\omega)^kp)(0)$. This is, of course, a version of Taylor's formula. \end{remark}

For $k\in\mathbb N$ and $T_k\in (V^{\odot  k})^*$, we define a linear operator $\langle T_k,Q^{k}\rangle \in\mathcal L(\mathcal P(V^*))$ by 
\begin{equation}\label{crtst5e5}
\big(\langle T_k,Q^{k}\rangle  p\big)(\omega):=\big\langle T_k,(Q^{k}p)(\omega)\big\rangle,\quad p\in\mathcal P(V^*),\ \omega\in V^*.
\end{equation} 
By \eqref{ge4ew34wq43w} and \eqref{crtst5e5},
\begin{align}
\big(\langle T_k, Q^{k}\rangle \langle P_n(\cdot),f_n\rangle\big)(\omega)&=
(n)_k\big\langle T_k,\langle  P_{n-k}(\omega),f_n\rangle\big\rangle\label{treseaw4aq42q24}\\
&=(n)_k\langle P_{n-k}(\omega)\odot T_k,f_n\rangle.\label{csesea43}
\end{align}

\begin{theorem}[Operator expansion]\label{bgnfm758} Let $P\in\mathbb P(V)$ and let $(P_n(\omega))_{n=0}^\infty$ be the corresponding polynomial sequence. Assume that 
  $P_0=1$ and $P_n (0)=0$ for all $n\geq 1$.    Assume that the corresponding lowering operators $Q(\theta)$  ($\theta\in V^*$) are shift-invariant. Let $T \in \mathcal L(\mathcal P  (V^*))$. Then $T$ is shift-invariant, i.e., $T\in\mathcal{SI}(\mathcal P  (V^*))$, if and only if there exist $T_0\in \mathbb F$ and $T_k\in (V^{\odot  k})^*$ for $k\in \mathbb N$ such that
\begin{equation}\label{vgxds5ese5}T=T_0\mathbf 1+\sum_{k=1}^\infty \langle T_k,
Q^{k}\rangle .\end{equation}
In the latter case,  
\begin{equation}\label{gxes5seasew}T_0=T 1,\quad \langle 
T_k, f_k \rangle =\frac1{k!}\big(T\langle P_k(\cdot),f_k\rangle\big)(0)\quad\text{for $k\in\mathbb N$ and $f_k\in V^{\odot   k}$}.
\end{equation} 
\end{theorem}

\begin{proof} For each polynomial $p\in\mathcal P(V^*)$ of degree $n$, $\langle T_k,
Q^{k}\rangle  p=0$ for all $k\ge n+1$. Therefore, the operator $T$ in formula \eqref{vgxds5ese5} is well-defined.

Since the lowering operators are shift-invariant, we have, for any $\theta,\eta,\omega \in V^*$ and $p\in\mathcal P(V^*)$,
\begin{align}
&\big\langle\eta^{\otimes k},(Q^{k}E(\theta)p)(\omega)\big\rangle=\big(Q(\eta)^kE(\theta)p\big)(\omega)=\big(E(\theta)Q(\eta)^k p\big)(\omega)\notag\\
&\quad =\big( Q(\eta)^k p\big)(\omega+\theta)=\langle\eta^{\otimes k},(Q^{k}p)(\omega+\theta)\rangle.\label{xstesw5w56u4}
\end{align}
By  \eqref{xstesw5w56u4} and Lemma~\ref{equation7867900},
\begin{equation}\label{xdsrearewaq4tq4t}
(Q^{k}E(\theta)p)(\omega)=(Q^{k}p)(\omega+\theta).
\end{equation}
By \eqref{xdsrearewaq4tq4t}, for each $T_k\in(V^{\odot k})^*$,
\begin{align} (\langle T_k,
Q^{k}\rangle E(\theta)p)(\omega)&=\big\langle T_k, (Q^{k}E(\theta)p)(\omega)\big\rangle=
\big\langle T_k,(Q^{k}p)(\omega+\theta)\big\rangle\notag\\
&=(\langle T_k,
Q^{k}\rangle p)(\omega+\theta)=(E(\theta)\langle T_k,
Q^{k}\rangle p)(\omega).\label{vcrsa4wa4tq2y6}
\end{align}
Therefore, the operator $\langle T_k,
Q^{k}\rangle $ is shift-invariant. But this implies that the operator~$T$ in formula \eqref{vgxds5ese5}  is shift-invariant.    

Let now $T \in \mathcal L(\mathcal P  (V^*))$ be shift-invariant, hence $T1$ is a constant by Lemma~\ref{rssw654w6u4weu6}. Let us prove that $T$ is as in formula~\eqref{vgxds5ese5}, where $(T_k)_{k=0}^\infty$  are given by \eqref{gxes5seasew}. Let $\omega,\theta \in V^*$ and  $p \in \mathcal P  (V^*)$. Applying Proposition~\ref{dyrde6e6e645}   to the polynomial $E(\omega)p$ and using \eqref{xdsrearewaq4tq4t}, we get:  
\begin{align}(E(\omega)p)(\theta)&=(E(\omega)p)(0)+\sum_{k=1}^\infty \frac{1}{k!}\,\big\langle P_k(\theta), (Q^{k}E(\omega)p)(0)\big\rangle\notag\\
&=p(\omega)+\sum_{k=1}^\infty \frac{1}{k!}\,\langle P_k(\theta), (Q^{k}p)(\omega)\rangle.\label{xsew5yw}
\end{align}
By using \eqref{xsew5yw} and the shift-invariance of the operator $T$, we get:
\begin{align}
(Tp)(\omega+\theta)&=(E(\omega)Tp)(\theta)=(TE(\omega)p)(\theta)\notag\\
&=
p(\omega)T1+\sum_{k=1}^\infty \frac{1}{k!}\,
\big(T\langle  P_k(\cdot), (Q^{k}p)(\omega)\rangle\big)(\theta).\label{vftftyu7}
\end{align}
Setting $\theta=0$ in \eqref{vftftyu7} and using \eqref{gxes5seasew}, we obtain:
\begin{align*}
(Tp)(\omega)&=p(\omega))T1+\sum_{k=1}^\infty \frac{1}{k!}\,
\big(T\langle  P_k(\cdot), (Q^{k}p)(\omega)\rangle\big)(0)
\\
&=T_0p(\omega)+\sum_{k=1}^\infty \frac{1}{k!}\,\langle T_k,(Q^{k}p)(\omega)\rangle\\
&=T_0p(\omega)+\sum_{k=1}^\infty \frac{1}{k!}\,(\langle T_k,
Q^{k}\rangle p)(\omega).\qedhere \end{align*}
\end{proof}



\begin{corollary}\label{567ty7}

{\rm (i)\/} For  $T \in \mathcal {SI}(\mathcal P  (V^*))$ of the form \eqref{vgxds5ese5}, define a formal tensor power series $(\mathcal IT)(v)\in \mathcal F(V,\mathbb F)$ by
$(\mathcal IT)(v):=T_0+\sum_{k=1}^\infty \langle T_k,v^{\otimes k}\rangle$.
Then $\mathcal I:\mathcal{SI}(\mathcal P  (V^*))\to \mathcal F(V,\mathbb F)$
 is an algebra isomorphism. In particular, any two shift-invariant operators commute. 

{\rm (ii)\/} A shift-invariant operator $T\in \mathcal{SI}(\mathcal P    (V^*))$  is invertible if and only if $T1\ne0$. In the latter case, $T^{-1}$ is also shift-invariant. 
\end{corollary}

\begin{proof} (i) By Theorem \ref{bgnfm758}, the map $\mathcal I$ is bijective. Since $\mathcal I$ is linear, it is therefore sufficient to prove that, for any $T^{(1)}_k\in (V^{\odot  k})^*$ and  $T^{(2)}_l\in (V^{\odot  l})^*$ ($k,l\in\mathbb N$), we have
$$\langle T^{(1)}_k,Q^{k}\rangle \langle T^{(2)}_l,Q^{l}\rangle 
=\langle T^{(1)}_k\odot T^{(2)}_l,Q^{k+l}\rangle. $$ 
But this immediately follows from \eqref{treseaw4aq42q24} and \eqref{csesea43}. The commutativity of the product in $\mathcal F(V;\mathbb F)$ implies the commutation of any two shift-invariant operators.

(ii) If $T1=0$, then $Tc=0$ for any constant $c$; hence the operator $T$ is not injective. If $T1\ne0$, then the existence of $T^{-1}\in  \mathcal{SI}(\mathcal P    (V^*)) $ follows from the fact that $\mathcal IT\in\mathcal F_0(V)$.
\end{proof}

  Below, for a formal power series $\mathcal T (v)=T_0+\sum_{k=1}^\infty \langle T_k,v^{\otimes k}\rangle\in\mathcal F(V,\mathbb F)$, we will denote $\mathcal T(Q):=\mathcal I^{-1}(\mathcal T(v))\in \mathcal{SI}(\mathcal P  (V^*))$. Thus,
  $\mathcal T(Q)=T_0\mathbf 1+\sum_{k=1}^\infty\langle T_k,Q^k\rangle$.

Let $G\in\mathcal L\big(\mathcal P(V^*),\mathcal P(V^*,V)\big)$. We will say that the operator $G$ is {\it shift-invariant\/} if, for each $p\in\mathcal P(V^*)$ and $\omega,\theta\in V^*$, we have $(GE(\theta)p)(\omega)=(Gp)(\omega+\theta)$.  We will denote by $ \mathcal{SI}\big(\mathcal P    (V^*),\mathcal P(V^*,V)\big)$ the set of all shift-invariant operators from $\mathcal L\big(\mathcal P(V^*),\mathcal P(V^*,V)\big)$.

Similarly to \eqref{crtst5e5}, for $k\in\mathbb N$ and 
 $G_k\in\mathcal L(V^{\odot k},V)$, we define a linear operator $ G_kQ^{k}\in \mathcal L\big(\mathcal P(V^*),\mathcal P(V^*,V)\big)$ by 
\begin{equation}\label{tsxt6uei76}
 \big(G_kQ^{k} p\big)(\omega):=  G_k\big((Q^{k}p)(\omega)\big),\quad p\in\mathcal P(V^*).
 \end{equation}
By \eqref{ge4ew34wq43w} and \eqref{tsxt6uei76},
\begin{equation}\label{dzsrea4ywu565468l}
\big(G_kQ^k\langle P_n(\cdot),v^{\otimes n}\rangle\big)(\omega)=(n)_k(G_kv^{\otimes k})\langle P_{n-k}(\omega),v^{\otimes(n-k)}\rangle. 
\end{equation}

\begin{corollary}\label{xesa4aq46q327y} 
{\rm (i)} Assume that the  lowering operators $Q(\theta)$  ($\theta\in V^*$) are as in Theorem~\ref{bgnfm758}.  Let $G \in \mathcal L\big(\mathcal P  (V^*),\mathcal P  (V^*,V)\big)$. Then  $G\in\mathcal{SI}\big(\mathcal P  (V^*),\mathcal P(V^*,V)\big)$ if and only if there exist $G_0\in V$ and $G_k\in \mathcal L(V^{\odot k},V)$ for $k\in \mathbb N$ such that
\begin{equation}\label{utf6d6eid}
G=G_0\mathbf 1+\sum_{k=1}^\infty G_k
Q^{k}.\end{equation}
In the latter case,  
$G_0=G 1$ and $ G_k f_k  =\frac1{k!}\big(G\langle P_k(\cdot),f_k\rangle\big)(0)$ for $k\in\mathbb N$ and $f_k\in V^{\odot   k}$.

{\rm (ii)} For   $G\in\mathcal{SI}\big(\mathcal P  (V^*),\mathcal P(V^*,V)\big)$ of the form \eqref{utf6d6eid}, define a formal tensor power series $(\mathcal JG)(v)\in \mathcal F(V,V)$ by
$(\mathcal JG)(v)=G_0+\sum_{k=1}^\infty G_kv^{\otimes k}$.
Then the map
$$\mathcal J:\mathcal{SI}\big(\mathcal P  (V^*),\mathcal P(V^*,V)\big)\to \mathcal F(V,V)$$
 is bijective.
\end{corollary}

\begin{proof} Let $G_k\in \mathcal L(V^{\odot k},V)$ ($k\in \mathbb N$).  By formula~\eqref{xdsrearewaq4tq4t}, the operator $G_kQ^{k}$ is shift-invariant. Therefore, each operator $G$ as in formula \eqref{utf6d6eid} is shift-invariant. 

On the other hand, let $G \in  \mathcal{SI}\big(\mathcal P    (V^*),\mathcal P(V^*,V)\big)$. Let $\eta\in V^*$ and note that, for each $p\in\mathcal P(V^*,V)$, $p_\eta(\omega):=\langle\eta,p(\omega)\rangle$ is a polynomial from $\mathcal P(V^*)$. Hence, we may define $G_\eta\in\mathcal L(\mathcal P(V^*))$ by $(G_\eta p)(\omega):=\langle \eta, (Gp)(\omega)\rangle$. Obviously, $G_\eta\in \mathcal{SI}(\mathcal P    (V^*))$. By Theorem~\ref{bgnfm758}, $G_\eta=G_{\eta,0}\mathbf 1+\sum_{k=1}^\infty \langle G_{\eta,k},D^{k}\rangle$, where $G_{\eta,0}=G_\eta 1=\langle \eta,G1\rangle$ and 
\[
\langle G_{\eta,k}, f_k \rangle =\frac1{k!}\big(G_\eta\langle P_k(\cdot),f_k\rangle\big)(0)
=\frac1{k!}\,\big\langle\eta, \big(G\langle P_k(\cdot),f_k\rangle\big)(0)\big\rangle,\quad f_k\in V^{\odot   k},\ k\in\mathbb N.
\]  This immediately implies part (i) of the corollary. Part (ii) follows from part (i). 
\end{proof}

 For a formal tensor power series $\mathcal G(v)=G_0+\sum_{k=1}^\infty G_kv^{\otimes k}\in \mathcal F(V,V)$, we denote  $\mathcal G(Q):=\mathcal J^{-1}(\mathcal G(v))\in \mathcal{SI}\big(\mathcal P  (V^*),\mathcal P(V^*,V)\big)$. Thus, $\mathcal G(Q)=G_0\mathbf 1+\sum_{k=1}^\infty G_kQ^k$.

\section{Umbral operators and polynomial sequences of binomial type}\label{cgfdxyrdrek}

Let $P\in\mathbb P(V)$. We say that $P$ is an {\it umbral operator} and the corresponding polynomial sequence $(P_n(\omega))_{n=0}^\infty$ is a {\it polynomial sequence of binomial type} or just a {\it binomial sequence}, if $P_0=1$ and for all $\omega, \theta \in V^*$,
\begin{equation}\label{xSSX65Q}P_n(\omega+\theta)=\sum_{k=0}^n{{n}\choose{k}}P_k(\omega)\odot P_{n-k}(\theta),\quad n\in\mathbb N.\end{equation}
We denote by $\mathbb B(V)$ the set of all umbral operators.

Obviously, $P=\mathbf 1$, the identity operator, is an umbral operator and the sequence of monomials $(\omega^{\otimes n})_{n=0}^\infty$ is a binomial sequence.

\begin{remark}Condition \eqref{xSSX65Q} is equivalent to requiring, for all $\omega,\theta\in V^*$ and  $v\in V$,  
\begin{equation*}\langle P_n(\omega+\theta),v^{\otimes n}\rangle=\sum_{k=0}^n{{n}\choose{k}}\langle 
P_k(\omega),v^{\otimes k}\rangle\langle P_{n-k}(\theta),v^{\otimes(n-k)}\rangle,\quad n\in\mathbb N.
\end{equation*}
\end{remark}

\begin{lemma}
\label{gfyrdr6ew64ui}
Let $(P_n(\omega))_{n=0}^\infty$ be a binomial sequence. Then $P_n(0)=0$ for all $n\in\mathbb N$.
\end{lemma}

\begin{proof}
For $n=1$, $P_1(\omega+\theta)=P_1(\omega)+P_1(\theta)$.  Setting $\omega=\theta=0$, we get
$P_1(0)=2P_1(0)$, hence $P_1(0)=0$. Now assume that $P_1(0)=P_2(0)=\dots=P_n(0)=0$. Then, similarly to the case $n=1$, we conclude that $P_{n+1}(0)=2P_{n+1}(0)$, hence $P_{n+1}(0)=0$. 
\end{proof} 

We will now present equivalent characterizations of a binomial sequence.

\begin{theorem}\label{45678765g} Let $P=[P_{kn}]_{k,n\in \mathbb N_0}\in \mathbb P(V)$, and let $(P_n(\omega))_{n=0}^\infty$ be the corresponding polynomial sequence. Assume that $P_0=1$ and $P_n(0)=0$ for all $n\in\mathbb N$.  Let $(Q(\theta))_{\theta \in V^*}$ be the corresponding lowering operators, and $Q$ the lowering differential. Then the following conditions are equivalent:

{\rm (B1)} The $(P_n(\omega))_{n=0}^\infty$ is a binomial sequence. 

{\rm(B2)} The lowering differential $Q$ is shift-invariant; equivalently for each $\theta \in V^*$, $Q(\theta)$ is shift-invariant.

{\rm (B3)} There exists a formal tensor power series $L(v)=\sum_{k=1}^\infty L_k v^{\otimes k}\in\mathcal F_1(V)$ such that 
\begin{equation}\label{vcftdtrdsy6ek}
Q=L(D);\end{equation} 
equivalently, for each $\theta \in V^*$, $Q(\theta)=\sum_{k=1}^\infty \langle \theta L_k,D^{k}\rangle$.

{\rm(B4)} The exponential generating function of the polynomial sequence $(P_n(\omega))_{n=0}^\infty$  
 has the form 
\begin{equation}\label{trse5sw5w}
\sum_{n=0}^\infty \frac{1}{n!}\,\langle P_n(\omega),v^{\otimes n}\rangle= \exp\big[\langle \omega,B(v)\rangle\big],\quad \omega \in V^*,
\end{equation}
Here 
where $B(v)\in \mathcal F_1(V)$. 

{\rm(B5)} There exists $B(v)\in\mathcal F_1(V)$ such that, for each $k\in \mathbb N$, 
\begin{equation}\label{567yuj8u}\sum_{n=k}^\infty\frac{k!}{n!}\, P_{kn}v^{\otimes n}=B(v)^{\otimes k}.\end{equation}
Formula \eqref{567yuj8u} is understood as an equality of formal tensor power series from $\mathcal F(V,V^{\odot k})$.

Furthermore, the formal tensor power series $B(v)\in\mathcal F_1(V)$ in  {\rm (B4)} and {\rm (B5)} are the same, while $L(v)\in\mathcal F_1(V)$ in  {\rm (B3)} is the compositional inverse of $B(v)$, i.e., $L(v)=B^{\langle-1\rangle}(v)$.
\end{theorem}

\begin{proof} We start with 

{\it Proof of\/} (B1)$\Rightarrow$(B2). We observe that, for any $\eta, \theta \in V^*$,
$E(\theta)Q(\eta)1=E(\theta)0=0$ and $Q(\eta)E(\theta)1=Q(\eta)1=0$,
 hence $E(\theta)Q(\eta)1=Q(\eta)E(\theta)1$. Now, let $n\in\mathbb N$ and $v\in V$. Using (B1) and formula $k\binom nk=n\binom{n-1}{k-1}$  for $k=1,2,\dots,n$, we obtain: 
\begin{align*}
&\big(Q(\eta)E(\theta)\langle P_n(\cdot),v^{\otimes n}\rangle\big)(\omega)
=\sum_{k=0}^n{n \choose k}\langle P_{n-k}(\theta), v^{\otimes(n-k)}\rangle\, \big(Q(\eta)\langle P_k(\cdot),v^{\otimes k}\rangle\big)(\omega)\\
&\quad =\sum_{k=1}^n{n \choose k}\langle P_{n-k}(\theta),v^{\otimes(n-k)}\rangle k\langle \eta,v \rangle \langle P_{k-1}(\omega),v^{\otimes(k-1)}\rangle\\
&\quad=n\langle \eta,v \rangle \sum_{k=1}^n{n-1 \choose k-1}\langle P_{n-k}(\theta),v^{\otimes(n-k)}\rangle \langle P_{k-1}(\omega),v^{\otimes(k-1)}\rangle\\
&\quad=n\langle \eta,v \rangle \sum_{k=0}^{n-1}{n-1 \choose k}\langle P_{n-1-k}(\theta),v^{\otimes(n-1-k)}\rangle \langle P_k(\omega),v^{\otimes k}\rangle\\
&\quad=n\langle \eta,v \rangle \langle P_{n-1}(\theta+\omega),v^{\otimes(n-1)}\rangle\\
&\quad=n\langle \eta,v \rangle \big(E(\theta) \langle P_{n-1}(\cdot),v^{\otimes(n-1)}\rangle\big)(\omega)\\
&\quad=\big(E(\theta)Q(\eta) \langle P_{n}(\cdot),v^{\otimes(n)}\rangle\big)(\omega).
\end{align*}
Hence, by Lemma~\ref{vcfcfydy} (i), we have $E(\theta)Q(\eta)p=Q(\eta)E(\theta)p$ for all $p\in \mathcal P  (V^*)$, i.e., $Q(\theta)$ is shift-invariant for each $\theta\in V^*$. This also implies that the operator $Q$ is shift-invariant.

{\it Proof of\/} (B2)$\Rightarrow$(B1).  By the assumption of the theorem, $P_0=1$, so we only need to prove that \eqref{xSSX65Q} holds. Let $\omega,\theta\in V^*$ and $f_n\in V^{\odot n}$. By formulas~\eqref{csesea43}, \eqref{vcrsa4wa4tq2y6}, Proposition~\ref{dyrde6e6e645} and Theorem~\ref{bgnfm758}, we obtain 
\begin{align*}
&\langle P_n(\omega+\theta),f_n\rangle=\big(E(\omega)\langle P_n(\cdot),f_n\rangle\big)(\theta)\\
&\quad= \big(E(\omega)\langle P_n(\cdot),f_n\rangle\big)(0)+\sum_{k=1}^n\frac1{k!}\,\big\langle P_k(\theta),\big(Q^{k}E(\omega)\langle P_n(\cdot),f_n\rangle\big)(0)\big\rangle\\
&\quad= \langle P_n(\omega),f_n\rangle+\sum_{k=1}^n\frac1{k!}\, 
\big(\langle P_k(\theta),Q^{k}\rangle  E(\omega)\langle P_n(\cdot),f_n\rangle\big)(0)\\
&\quad= \langle P_n(\omega),f_n\rangle+\sum_{k=1}^n\frac1{k!}\, 
\big(E(\omega)\langle P_k(\theta),Q^{k}\rangle  \langle P_n(\cdot),f_n\rangle\big)(0)\\
&\quad= \langle P_n(\omega),f_n\rangle+\sum_{k=1}^n\frac1{k!}\, 
\big(\langle P_k(\theta),Q^{k}\rangle  \langle P_n(\cdot),f_n\rangle\big)(\omega)\\
&\quad= \langle P_n(\omega),f_n\rangle+\sum_{k=1}^n\binom n k\big\langle P_{n-k}(\omega)\odot P_k(\theta),f_n\big\rangle. 
\end{align*}

{\it Proof of\/} (B2)$\Rightarrow$(B3). In Corollary~\ref{xesa4aq46q327y} choose $Q=D$, and apply this result to the shift-invariant operator $Q$. Noting that $Q1=0$, we conclude that  $Q=\sum_{k=1}^\infty L_kD^{k}$, where 
$$L_kf_k=\frac1{k!}\big(Q\langle\cdot^{\otimes k},f_k\rangle\big)(0),\quad f_k\in V^{\odot k},\ k\in\mathbb N.$$

It remains to prove that the operator $L_1\in\mathcal L(V)$ is invertible. Since $P_1(0)=0$, we have, by \eqref{xreaw4a4y}--\eqref{gftdtrst5}, $\langle P_1(\omega),v\rangle=\langle\omega,P_{11}v\rangle$ for $ v\in V$. 
Since the operator $P_{11}$ is invertible, we obtain $\langle\omega,v\rangle=\langle P_1(\omega), P^{-1}_{11}v\rangle$.
 Hence, $L_1=P_{11}^{-1}$, which is an invertible operator.

{\it Proof of\/} (B3)$\Rightarrow$(B2).  Immediate by Corollary \ref{xesa4aq46q327y}.

{\it Proof of\/} (B3)$\Rightarrow$(B4). Below we will use the isomorphism  $\mathcal I:\mathcal{SI}(\mathcal P  (V^*))\to \mathcal F(V;\mathbb F)$ from Corollary~\ref{567ty7}~(i) with $Q=D$. 
We will divide this proof into three steps.

\textit{Step 1}.  Let $\omega\in V^*$. By Lemma~\ref{fghj9876b}, 
\begin{equation}\label{fgtxtdstrs}(\mathcal IE(\omega))(v)=\exp[\langle\omega,v\rangle],
\end{equation} 
and by (B3),  
$$(\mathcal IQ(\omega))(v)=\sum_{k=1}^\infty\langle \omega L_k,v^{\otimes k}\rangle=\langle\omega,L(v)\rangle.$$ 
Therefore,  for each $k\in\mathbb N$,
\begin{equation}\label{gvyfukfgtdr}
(\mathcal IQ(\omega)^k)(v)=\langle\omega,L(v)\rangle^k=\langle\omega ^{\otimes k},L(v)^{\otimes k}\rangle.\end{equation}

\textit{Step 2}. Let us fix $T_k\in(V^{\odot k})^*$. By (B3) and Theorem~\ref{bgnfm758}, the operator $\langle T_k,Q^{k}\rangle$ is shift-invariant and $\langle T_k,Q^{k}\rangle=R_0+\sum_{n=1}^\infty \langle R_n,D^{n}\rangle$. Here  $R_0=\langle T_k, Q^{k}\rangle1=0$ and for $n\in\mathbb N$ and $f_n\in V^{\odot n}$, 
\begin{align}
\langle R_n,f_n\rangle&=
\frac{1}{n!}\big(\langle T_k,Q^{k}\rangle \langle \cdot^{\otimes n},f_n\rangle\big)(0)\notag\\
&=\frac1{n!}\big\langle T_k, \big(Q^{k}\langle \cdot^{\otimes n},f_n\rangle\big)(0)\big\rangle=\frac1{n!}\langle T_k , L_{k,n}f_n\rangle,\notag
 \end{align}
 where $L_{k,n}\in \mathcal L(V^{\odot n},V^{\odot k})$ is given by 
 $$L_{k,n}f_n:=
 \big(Q^{k}\langle \cdot^{\otimes n},f_n\rangle\big)(0).$$
  Observe that $L_{kn}=0$ for $k>n$. 
  Hence,
\begin{equation}\label{fserts5bvgc}
  \big(\mathcal I\langle T_k,Q^{k}\rangle\big)(v)=\sum_{n=k}^\infty\frac1{n!}\, \langle T_k , L_{k,n}v^{\otimes n}\rangle=\langle T_k,L^{(k)}(v)\rangle,\end{equation}
  where $L^{(k)}(v):=\sum_{n=k}^\infty\frac1{n!}\, L_{k,n}v^{\otimes n}\in\mathcal F(V,V^{\odot k})$. In particular, setting $T_k=\omega^{\otimes k}$ with $\omega\in V^*$, we obtain from  \eqref{fserts5bvgc}:
  \begin{equation}\label{vdtje6}
  (\mathcal I Q(\omega)^k)(v)=\langle \omega^{\otimes k},L^{(k)}(v)\rangle. 
  \end{equation}
Comparing formulas  \eqref{gvyfukfgtdr} and \eqref{vdtje6},  and using Lemma~\ref{equation7867900}, we conclude that 
$L^{(k)}(v)=L(v)^{\otimes k}$. Hence, by \eqref{fserts5bvgc}, for each $T_k\in (V^{\odot k})^*$, we have 
\begin{equation}\label{ctxsts}
\big(\mathcal I\langle T_k,Q^{k}\rangle \big)(v)=\langle T_k,L(v)^{\otimes k}\rangle.
 \end{equation}
 
 {\it Step 3}. Let $\omega\in V^*$. The application of Theorem~\ref{bgnfm758} to the shift operator $E(\omega)$ gives: $E(\omega) = \mathbf 1+ \sum_{k=1}^\infty \langle T_k,Q^{k}\rangle$,
where for $k\in\mathbb N$, 
\[\langle T_k, f_k\rangle =\frac{1}{k!}\big(E(\omega)\langle P_k(\cdot),f_k\rangle\big)(0)=\frac{1}{k!}\langle P_k(\omega),f_k\rangle, \quad f_k\in V^{\odot k}.\]
Therefore, $E(\omega) = \mathbf 1+ \sum_{k=1}^\infty\frac1{k!}\,\langle P_k(\omega), Q^{k}\rangle$.
Hence, by \eqref{ctxsts},
\begin{align}
(\mathcal IE(\omega))(v)& =  1+ \sum_{k=1}^\infty \frac1{k!}\,\big(\mathcal I
\langle P_k(\omega), Q^{k}\rangle \big)(v)\label{56gh56fg8s}\\
& =  1+ \sum_{k=1}^\infty \frac1{k!}\,\langle P_k(\omega),L(v)^{\otimes k}\rangle.
\label{cftdtyrdrd}
\end{align}
Since the sum on the right-hand side of formula \eqref{56gh56fg8s} is infinite, this formula requires a justification.  But this can be easily done  if one takes into account that $\langle P_k(\omega), Q^{k}\rangle \langle \cdot^{\otimes n},v^{\otimes n}\rangle=0$ for $k\ge n+1$. 

By \eqref{fgtxtdstrs} and \eqref{cftdtyrdrd}, we obtain 
$$1+ \sum_{k=1}^\infty \frac1{k!}\,\langle P_k(\omega),L(v)^{\otimes k}\rangle= \exp[\langle\omega,v\rangle],$$
and so 
$$1+ \sum_{k=1}^\infty \frac1{k!}\,\langle P_k(\omega),v^{\otimes k}\rangle= \exp\big[\langle\omega,L^{\langle-1\rangle}(v)\rangle\big].$$
Thus, formula \eqref{trse5sw5w} holds with $B(v):=L^{\langle-1\rangle}(v)$.

{\it Proof of\/} (B4)$\Rightarrow$(B1). For any $\omega,\theta\in V^*$, we have, by (B4),
\begin{align*}
&\sum_{n=0}^\infty \frac{1}{n!}\,\langle P_n(\omega+\theta),v^{\otimes n}\rangle=\exp[\langle \omega+\theta, B(v)\rangle]\\
&\quad =\exp[\langle \omega, B(v)\rangle]\exp[\langle \theta, B(v)\rangle]\\
&\quad=\bigg(\sum_{n=0}^\infty \frac{1}{n!}\,\langle P_n(\omega),v^{\otimes n}\rangle\bigg)\bigg(\sum_{m=0}^\infty \frac{1}{m!}\,\langle P_m(\theta),v^{\otimes m}\rangle\bigg)\\
&\quad=1+\sum_{n=1}^\infty \frac1{n!}\sum_{k=0}^n {n \choose k}\big\langle P_k(\omega)\odot P_{n-k}(\omega), v^{\otimes n}\big\rangle.
\end{align*}
which implies (B1). 

{\it Proof of\/} (B4)$\Rightarrow$(B5). Condition (B4) can be equivalently formulated as follows: for fixed $\omega \in V^*$ and $v\in V$, we have
\begin{equation*}
\sum_{n=0}^\infty \frac {t^n}{n!}\,\langle P_n(\omega),v^{\otimes n}\rangle=\exp \bigg[\sum_{n=1}^\infty t^n\langle \omega,B_n v^{\otimes n}\rangle \bigg],\label{67gh76}\end{equation*}
where the above equality is understood as an equality of formal power series in $t \in \mathbb F$. Hence, for each $s\in \mathbb F$, we have 
\begin{equation}\label{67gh895b}
\sum_{n=0}^\infty \frac {t^n}{n!}\,\langle P_n(s\omega),v^{\otimes n}\rangle=\exp \bigg[s\sum_{n=1}^\infty t^n\langle \omega,B_n v^{\otimes n}\rangle \bigg].
\end{equation}
By \eqref{gftdtrst5}, the left-hand side of equality \eqref{67gh895b} can be written as follows
\begin{align}
\sum_{n=0}^\infty \frac {t^n}{n!}\,\langle P_n(s\omega),v^{\otimes n}\rangle&=1+\sum_{n=1}^\infty \frac{t^n}{n!} \sum_{k=1}^n\langle (s\omega)^{\otimes k},P_{kn} v^{\otimes n}\rangle \notag\\
&=1+\sum_{n=1}^\infty \frac{t^n}{n!} \sum_{k=1}^ns^k\langle \omega^{\otimes k},P_{kn} v^{\otimes n}\rangle.\label{jk78g9jk}\end{align}
For fixed $\omega \in V^*$ and $v\in V$, $\sum_{k=1}^n s^k\langle \omega^{\otimes k}, P_{kn}v^{\otimes n}\rangle$
is a polynomial of degree $n$ in the variable $s\in \mathbb F$. Hence, the expression in \eqref{jk78g9jk} can be thought of as a formal power series in variables $s,t\in \mathbb F$. Therefore, by \eqref{67gh895b} and \eqref{jk78g9jk}, 
\begin{equation}1+\sum_{k=1}^\infty \sum_{n=k}^\infty\frac{t^ns^k}{n!}\langle \omega^{\otimes k},P_{kn} v^{\otimes n}\rangle= 1+\sum_{k=1}^\infty\frac1{k!}s^k\bigg(\sum_{n=1}^\infty t^n\langle \omega, B_nv^{\otimes n}\rangle \bigg)^k.
\label{vcfxtdxt}\end{equation}
which is understood as an equality of formal power series in variables $s,t\in \mathbb F$. Equating in \eqref{vcfxtdxt} the formal power series in variable $t$ by $s^k$, we obtain 
$$\sum_{n=k}^\infty \frac{t^n}{n!}\langle \omega^{\otimes k},P_{kn} v^{\otimes n}\rangle=\frac1{k!}\bigg(\sum_{n=1}^\infty t^n\langle \omega, B_nv^{\otimes n}\rangle \bigg)^k,\quad k\in\mathbb N.$$
Therefore, for each $k\in\mathbb N$, we get the following equality of formal tensor power series from $\mathcal F(V,\mathbb F)$:
\begin{equation*}\sum_{n=k}^\infty\langle \omega^{\otimes k},\frac {k!}{n!}\,P_{kn}v^{\otimes n}\rangle = \bigg(\sum_{n=1}^\infty\langle \omega,B_nv^{\otimes n} \rangle\bigg)^k=\langle \omega, B(v)\rangle^k=\langle\omega^{\otimes k},B(v)^{\otimes k}\rangle. \end{equation*}
By Lemma~\ref{equation7867900}, this implies (B5).  
 
 {\it Proof of\/} (B5)$\Rightarrow$(B4). Straightforward: one just needs to reverse the arguments of the proof of (B4)$\Rightarrow$(B5). Thus, the theorem is proven.
\end{proof}

  \begin{corollary}\label{crts65sw6ue6u} (i) For each $B(v)\in\mathcal F_1(V)$,  there exists a binomial sequence $(P_n(\omega))_{n=0}^\infty$ with generating function~\eqref{trse5sw5w}.
  
  (ii) A linear operator $Q\in\mathcal L\big(\mathcal P(V^*),\mathcal P(V^*,V)\big)$
is the lowering differential    for a binomial sequence $(P_n(\omega))_{n=0}^\infty$
if and only if there exists $L(v)\in\mathcal F_1(V)$ such that   $Q=L(D)$.  
  \end{corollary}
  
  \begin{proof} (i) Formula~\eqref{trse5sw5w} determines a polynomial sequence $(P_n(\omega))_{n=0}^\infty$. Setting $\omega=0$ in formula~\eqref{trse5sw5w}, we get $\sum_{n=0}^\infty\frac1{n!}\langle P_n(0),v^{\otimes n}\rangle=1$. Hence, $P_0=1$ and $P_n(0)=0$ for $n\in\mathbb N$. Therefore, by Theorem~\ref{45678765g} (B4), the polynomial sequence $(P_n(\omega))_{n=0}^\infty$ is binomial. 
  
 (ii)  Let $L(v)\in\mathcal F_1(V)$ and define $B(v):=L^{\langle-1\rangle}(v)\in\mathcal F_1(V)$. By part~(i), there exists a binomial sequence $(P_n(\omega))_{n=0}^\infty$ with generating function~\eqref{trse5sw5w}. 
By Theorem~\ref{45678765g} (B3), the  lowering differential  for this binomial sequence has the form  $Q=L(D)$. The converse statement is obvious.  
  \end{proof}

Recall that a set partition $\pi=\{\mathcal A_1,\dots, \mathcal A_k\}$ of a finite set $\mathcal X\neq \varnothing$ is an unordered collection of disjoint nonempty subsets (parts) of $\mathcal X$ whose union is $\mathcal X$. We will denote by $\mathcal P(\mathcal X)$ the collection of all set partitions of $\mathcal X$.  For a set partition $\pi\in\mathcal P(\mathcal X)$, we denote by $|\pi|$ the number of sets in $\pi$. For a set $\mathcal A\in\pi$, we denote by $|\mathcal A|$, the number of elements of the set $\mathcal A$.  In the case $\mathcal X=\{1,2,\dots,n\}$, we denote $\mathcal P(n):=\mathcal P(\mathcal X)$.

The following corollary extends \cite[Proposition 5.1]{FKLO}. 

\begin{corollary}\label{cxtsxtstst} Let  $(P_n(\omega))_{n=0}^\infty$ be a binomial sequence with generating function \eqref{trse5sw5w}, with $B(v)=\sum_{k=1}^\infty\frac1{k!}\,\tilde B_kv^{\otimes k}$.  Then, for any $\omega\in V^*$,  $n\in\mathbb N$, and $v_1,\dots,v_n\in V$, we have
\begin{equation}\label{fgy653d}
\langle P_n(\omega),v_1\odot\dots\odot v_n\rangle=
\sum_{\pi\in\mathcal P(n)}\prod_{\mathcal A\in\pi}\Big\langle\omega,\tilde B_{|\mathcal A|}\Big(\underset{i\in \mathcal A}\odot v_i\Big)\Big\rangle.\end{equation}
In particular, for $\omega\in V^*$, $v\in V$, and $n\in\mathbb N$,
\begin{equation}\label{rdtrs6u}
\langle P_n(\omega),v^{\otimes n}\rangle=
\sum_{\pi\in\mathcal P(n)}\prod_{\mathcal A\in\pi}\langle\omega,\tilde B_{|\mathcal A|}v^{\otimes|\mathcal A|}\rangle.\end{equation}
  \end{corollary}
  
  \begin{proof}It follows from Theorem~\ref{45678765g} (B5) that
  \begin{align}P_{kn}&=\frac{n!}{k!}\sum_{l_1,\dots,l_k\geq 1\atop l_1+\cdots+l_k=n}\frac{1}{l_1!\cdots l_k!}\,\tilde B_{l_1}\odot\tilde B_{l_2}\odot\dots \odot\tilde B_{l_k}\notag\\
&=\sum_{\substack{(j_1,j_2,\dots,j_n)\in \mathbb {N}_0^n\\ j_1+j_2+\cdots+j_n=k,\\ j_1+2j_2+\cdots+nj_n=n}}\frac{n!}{j_1!j_2!\cdots j_n!(1!)^{j_1}(2!)^{j_2}\cdots(n!)^{j_n}}\tilde B_1^{\odot j_1}\odot\tilde B_2^{\odot j_2}\odot\dots \odot\tilde B_n^{\odot j_n}\notag\\
&=\sum_{\pi=\{\mathcal A_1,\dots,\mathcal A_k\}\in\mathcal P(n,k)}\tilde B_{|\mathcal A_1|}\odot\tilde B_{|\mathcal A_2|}\odot\cdots \odot \tilde B_{|\mathcal A_k|},\label{cfxtxstz}\end{align}
where $\mathcal P(n,k)$ denotes the collection of all set partitions  from $\mathcal P(n)$
 that have exactly $k$ parts. 
In view of \eqref{xreaw4a4y}--\eqref{gftdtrst5}, formula \eqref{cfxtxstz} implies \eqref{rdtrs6u}. The latter formula yields, in turn, \eqref{fgy653d}.  
  \end{proof}

 \section{Sheffer sequences}\label{cxzsrarewtytyrf}

 Let $S\in\mathbb P(V)$. We say that $S$ is a {\it Sheffer operator} and the corresponding polynomial sequence $(S_n(\omega))_{n=0}^\infty$ is a {\it Sheffer sequence} if $(S_n(\omega))_{n=0}^\infty$ has  the generating function of the form  
\begin{equation}\sum_{n=0}^\infty \frac1{n!}\langle S_n(\omega), v^{\otimes n}\rangle=\exp [\langle \omega, B(v)\rangle]A(v),\label{ldj84n}\end{equation} 
where $B(v)\in \mathcal F_1(V)$ and $A(v)\in \mathcal F_0(V)$. We denote by $\mathbb S(V)$ the set of Sheffer operators.

It follows from the definition of a Sheffer operator and Corollary~\ref{crts65sw6ue6u} (i)  that, to each Sheffer operator $S$, there corresponds a unique umbral operator $P$ whose 
binomial sequence $(P_n(\omega))_{n=0}^\infty$ has generating function \eqref{trse5sw5w}.

The following theorem gives  equivalent characterizatios of a Sheffer sequence.

\begin{theorem}\label{ctrst5u}
Let $P\in \mathbb B(V)$  and let $(P_n(\omega))_{n=0}^\infty$ be the corresponding binomial  sequence that has generating function \eqref{trse5sw5w}. Let $Q$ be the corresponding lowering differential.
Let $S=[S_{kn}]_{k,n\in\mathbb N_0}\in\mathbb P(V)$ and let $(S_n(\omega))_{n=0}^\infty$ be the corresponding polynomial  sequence.  The following conditions are equivalent:

{\rm (S1)} There exists  $A(v)\in \mathcal F_0(V)$ such that $(S_n(\omega))_{n=0}^\infty$ has generating function~\eqref{ldj84n}, hence $(S_n(\omega))_{n=0}^\infty$ is a Sheffer sequence.  

{\rm (S2)} Let a linear operator $T\in\mathcal L(\mathcal P(V^*))$ be defined by $T:=PS^{-1}$,
i.e.,  for each $n\in\mathbb N_0$ and $f_n\in V^{\odot n}$,
\begin{equation}T\langle S_n(\cdot),f_n\rangle=\langle P_n(\cdot),f_n\rangle.\label{2b3dh}\end{equation}
Then $T$ is shift-invariant.

{\rm (S3)} The $Q$ is the lowering differential  for the polynomial sequence $(S_n(\omega))_{n=0}^\infty$.

{\rm (S4)} For all $\omega,\theta \in V^*$  and $n\in \mathbb N$,
\begin{equation}\label{fgxtgxt}  S_n(\omega+\theta)=\sum_{k=0}^n {n \choose k}S_k(\omega)\odot P_{n-k}(\theta) .\end{equation}

{\rm (S5)} There exists $\rho(v)=\sum_{k=0}^\infty \langle \rho_k, v^{\otimes k}\rangle\in \mathcal F_0(V)$ such that, for all $\omega \in V^*$  and $n \in \mathbb N$,
\begin{equation}\label{guf6rd65s5}S_n(\omega) =\sum_{k=0}^n{n \choose k}
\rho_k\odot P_{n-k}(\omega).\end{equation}

{\rm (S6)} There exists $A(v)\in \mathcal F_0(V)$ such that, for each $k\in \mathbb N_0$, 
\begin{equation}\label{fgh561560}\sum_{n=k}^\infty \frac{k!}{n!}\,S_{kn}v^{\otimes n}= B(v)^{\otimes k}A(v).\end{equation}
Formula \eqref{fgh561560} is understood as an equality of formal tensor power series from $\mathcal F(V,V^{\odot k})$.

Furthermore, the formal tensor power series $A(v)=\sum_{k=0}^\infty\langle  A_k, v^{\otimes k}\rangle\in \mathcal F_0(V)$ in~{\rm (S1)} and~{\rm (S6)} are the same, the $\rho(v)$ in {\rm (S5)} has the form $\rho(v)=\sum_{k=0}^\infty \langle k!\, A_k, v^{\otimes k}\rangle$ (i.e., $\rho_k=k!\,A_k$), and the operator $T$ in {\rm (S2)} is given by 
$T=\tau(D)$, where $\tau(v):=A^{-1}(B^{\langle -1\rangle}(v))=A^{-1}(L(v))\in\mathcal F_0(V)$.
\end{theorem}

\begin{proof} We start with 

{\it Proof of\/} (S3)$\Rightarrow$(S2). By (S3) and formula \eqref{csesea43}, we have, for $R_k\in (V^{\odot k})^*$ and $v\in V$, 
\begin{align}
&T\langle R_k,Q^{k}\rangle \langle S_n(\cdot),v^{\otimes n}\rangle=
(n)_k\langle R_k,v^{\otimes k}\rangle T\langle S_{n-k}(\cdot),v^{\otimes (n-k)}\rangle\notag\\
&\quad=(n)_k\langle R_k,v^{\otimes k}\rangle \langle P_{n-k}(\cdot),v^{\otimes (n-k)}\rangle\notag\\
&\quad =\langle R_k,Q^{k}\rangle\langle P_n(\cdot),v^{\otimes n}\rangle=
\langle R_k,Q^{k}\rangle T\langle S_n(\cdot),v^{\otimes n}\rangle.\notag
\end{align} 
Hence, by Lemma~\ref{vcfcfydy} (i), the operators $T$ and $\langle R_k,Q^{k}\rangle$ commute. By Theorem~\ref{bgnfm758}, the operator $T$ commutes with any shift-invariant operator. In particular, $T$ commutes with any shift operator $E(\theta)$ ($\theta\in V^*$), i.e., $T\in \mathcal{SI}(\mathcal P(V^*))$. 

{\it Proof of\/} (S2)$\Rightarrow$(S3). By (S2),
$$T1=TS^{-1}_{00}S_{00}=S^{-1}_{00}TS_{00}=S^{-1}_{00}P_{00}=S^{-1}_{00}\neq 0.$$ Hence, by Corollary \ref{567ty7} (ii), the inverse operator $T^{-1}$ exists and is shift-invariant. By~Theorem~\ref{45678765g} (B2), for each $\theta \in V^*$, $Q(\theta)$ is shift-invariant. By Corollary \ref{567ty7} (i), the operators $T^{-1}$ and $Q(\theta)$ commute. Hence, for all $\theta \in V^*$, $v\in V$ and $n\in \mathbb N$,
\begin{align}
&Q(\theta)\langle S_n(\cdot),v^{\otimes n}\rangle=Q(\theta)T^{-1}\langle P_n(\cdot),v^{\otimes n}\rangle=T^{-1}Q(\theta)\langle P_n(\cdot),v^{\otimes n}\rangle\notag\\
&\quad =T^{-1}n\langle \theta,v \rangle \langle P_{n-1}(\cdot),v^{\otimes (n-1)}\rangle 
=n\langle \theta,v \rangle T^{-1} \langle P_{n-1}(\cdot),v^{\otimes (n-1)}\rangle \notag\\
&\quad =n\langle \theta,v \rangle \langle S_{n-1}(\cdot),v^{\otimes (n-1)}\rangle, \notag 
\end{align}
which implies (S3).

{\it Proof of \/} (S2)$\Rightarrow$(S1). Below we will again use the isomorphism  $\mathcal I:\mathcal{SI}(\mathcal P  (V^*))\to \mathcal F(V;\mathbb F)$ from Corollary~\ref{567ty7}~(i) with $Q=D$.

We already proved that, under assumption (S2), the operator $T$ is invertible and $T^{-1}$ is shift-invariant.  Let $\omega \in V^*$. Applying Theorem~\ref{bgnfm758}   to the shift-invariant operator $E(\omega)T^{-1}$,  we have
$E(\omega)T^{-1}=R_0\mathbf 1+\sum_{k=1}^\infty \langle R_k,Q^{k}\rangle$,
where $R_0=E(\omega)T^{-1}1=E(\omega)S_{00}=S_{00}$ and for $k\in \mathbb N$ and $f_k\in V^{\odot k}$,
\begin{align}
\langle R_k,f_k\rangle &=\frac1{k!}\big(E(\omega)T^{-1}\langle P_k(\cdot),f_k\rangle\big)(0)\notag\\
&=\frac1{k!}\big(E(\omega) \langle S_k(\cdot),f_k\rangle\big)(0)=\frac1{k!}\,\langle S_k(\omega),f_k\rangle,\notag
\end{align}
hence $R_k=\frac1{k!}\, S_k(\omega)$. Thus,
\begin{equation}\label{7777777k}E(\omega)
T^{-1}=S_{00}\mathbf 1+\sum_{k=1}^\infty \frac1{k!}\,\langle S_k(\omega),Q^{k}\rangle.\end{equation}

 By formula~\eqref{ctxsts}, 
\begin{equation}\label{dsrsa5uewu6}
\big(\mathcal I\langle S_k(\omega),Q^{k}\rangle\big)(v)=\langle S_k(\omega),L(v)^{\otimes k}\rangle.\end{equation}
Therefore, by \eqref{7777777k},
\begin{equation}\label{rfgh67ffd}
\big(\mathcal I E(\omega)T^{-1}\big)(v)=
S_{00}+\sum_{k=1}^\infty\frac1{k!}\langle S_k(\omega),L(v)^{\otimes k}\rangle=\sum_{k=0}^\infty\frac1{k!}\langle S_k(\omega),L(v)^{\otimes k}\rangle.
\end{equation}

We define $\tau(v)\in\mathcal F_0(V)$ by $\tau(v):=(\mathcal IT)(v)$. By Corollary~\ref{567ty7} (i) and formula \eqref{rfgh67ffd}, we have
\begin{align}
(\mathcal IE(\omega))(v)&=\big(\mathcal IE(\omega)T^{-1}T\big)(v)=\big(\mathcal IE(\omega)T^{-1}\big)(v)(\mathcal IT)(v)\notag\\
&=\bigg(\sum_{k=0}^\infty\frac1{k!}\langle S_k(\omega),L(v)^{\otimes k}\rangle\bigg)\tau(v). \label{fxdsxdesese}
\end{align}
 By \eqref{fgtxtdstrs} and \eqref{fxdsxdesese},
 \begin{equation}\label{vcfxtdstrest}
 \sum_{k=0}^\infty\frac1{k!}\langle S_k(\omega),L(v)^{\otimes k}\rangle=\exp[\langle\omega,v\rangle]\,\tau^{-1}(v).
 \end{equation} 
Recall that $L(v)=B^{\langle -1\rangle}(v)$. Hence, formula \eqref{vcfxtdstrest} implies 
\begin{equation}\label{fdxzrdzszsw}
\sum_{k=0}^\infty\frac1{k!}\langle S_k(\omega),v^{\otimes k}\rangle=\exp[\langle\omega,B(v)\rangle]\,\tau^{-1}(B(v))=\exp[\langle\omega,B(v)\rangle]A(v),\end{equation}
where 
\begin{equation}\label{vfsrar45w5y}A(v):=\tau^{-1}(B(v))\in\mathcal F_0(V).
\end{equation} 

{\it Proof of \/} (S1)$\Rightarrow$(S2). We define $r(v)\in \mathcal F(V;\mathbb F)$ by 
\begin{equation}\label{fcfctyxzrray}
r(v):=A(L(v))=A(B^{\langle -1\rangle}(v)).\end{equation}
  Since $r(0)=A(0)\ne0$, we conclude that $r(v)\in\mathcal F_0(V)$. We define $R\in \mathcal{SI}(\mathcal P(V^*))$ by $R:=\mathcal I^{-1}(r(v))$. By construction, the block-matrix $[R_{kn}]_{k,n\in\mathbb N_0}$ of the operator $R$ is infinite upper triangular, and for each $n\in\mathbb N_0$, $R_{nn}=r(0)\mathbf 1_n$. Hence, $R_{nn}$ is invertible, and so $R\in\mathbb P(V)$.

Next, we define $\tilde {S}=[\tilde {S}_{kn}] _{k,n\in\mathbb N_0}\in \mathcal L (\mathcal P(V^*))$ by $\tilde{S} := RP$. By Lemma~\ref{vcfcfydy} (ii), $\tilde S\in\mathbb P(V)$.
 Let $(\tilde S_n(\omega))_{n=0}^\infty$ be the corresponding polynomial sequence. 

 Applying Theorem \ref{bgnfm758} to the shift-invariant operator $E(\omega)R$, we have 
$E(\omega)R=H_0\mathbf 1+\sum_{k=1}^\infty \langle H_k, Q^{k}\rangle$,
where $H_0=E(\omega)R1=E(\omega)r(0)=r(0)=\tilde S_{00}$, and for $k\geq 1$,
\begin{align*}
\langle H_k,f_k\rangle&=\frac1{k!}\big(E(\omega)R\langle P_k(\cdot),f_k\rangle\big)(0)=\frac1{k!}\big(E(\omega)RP\langle\cdot^{\otimes k},f_k\rangle\big)(0)\\
& =\frac1{k!}\big(E(\omega)\tilde S\langle\cdot^{\otimes k},f_k\rangle\big)(0)=\frac1{k!}\,\langle\Tilde S_k(\omega),f_k\rangle.
\end{align*}
Thus, for $k\geq 1$, $H_k=\frac 1{k!}\,\tilde S_k(\omega)$, and so 
\begin{equation}E(\omega)R=\tilde S_{00}\mathbf 1+\sum_{k=1}^\infty \frac 1{k!}\langle \tilde S_k(\omega),
Q^{k}\rangle.\label{fghj987}\end{equation} 
Similarly to formulas \eqref{7777777k}--\eqref{fdxzrdzszsw}, we conclude from \eqref{fcfctyxzrray} and \eqref{fghj987} that 
\[ \sum_{k=0}^\infty\frac1{k!}\langle \tilde S_k(\omega),v^{\otimes k}\rangle=\exp\big[\langle \omega, B(v)\rangle\big]A(v).
\]
Hence, by (S1), $S=\tilde S=RP$. Therefore, $R=SP^{-1}$, and so $T=PS^{-1}=R^{-1}$. Hence, the operator $T$ is shift-invariant. 
 
 Thus, we have proved that the conditions (S1), (S2), and (S3) are equivalent.

{\it Proof of\/} (S2)$\Rightarrow$(S4). For $\omega, \theta \in V^*$, $v\in V$ and $n\in \mathbb N$, we have, by (S2) and Cor\-ol\-lary~\ref{567ty7}~(ii), 
\begin{align*}
\langle S_n(\omega+\theta), v^{\otimes n})\rangle&=\big(E(\theta)\langle S_n(\cdot), v^{\otimes n}\rangle)\big(\omega)=\big(E(\theta)T^{-1}\langle P_n(\cdot), v^{\otimes n}\rangle\big)(\omega)\\
& =\big(T^{-1}E(\theta)\langle P_n(\cdot), v^{\otimes n}\rangle\big)(\omega)\\
&=\sum_{k=0}^n {n \choose k}\big(T^{-1}\langle P_k(\cdot), v^{\otimes k}\rangle\big)(\omega)\langle P_{n-k}(\theta), v^{\otimes (n-k)}\rangle\\
&=\sum_{k=0}^n {n \choose k}\langle S_k(\omega), v^{\otimes k}\rangle\langle P_{n-k}(\theta), v^{\otimes (n-k)}\rangle,
\end{align*}
which implies (S4). 

{\it Proof of\/}(S4)$\Rightarrow$(S5). Setting in \eqref{fgxtgxt} $\theta=0$, we obtain \eqref{guf6rd65s5}, in which $\rho_k=S_k(0)$ ($k\in\mathbb N$).

{\it Proof of\/}(S5)$\Rightarrow$(S3). For $\theta\in V^*$, $v\in V$ and $n\in \mathbb N$, we have, by (S5),
\begin{align*}
&Q(\theta)\langle S_n(\cdot),v^{\otimes n}\rangle=\sum_{k=0}^n{n \choose k}\langle\rho_k,v^{\otimes k}\rangle\, Q(\theta) \langle P_{n-k}(\cdot),v^{\otimes(n-k)}\rangle\\
&\quad=\sum_{k=0}^{n-1}{n \choose k}\langle\rho_k,v^{\otimes k}\rangle (n-k)\langle\theta,v\rangle \langle P_{n-k-1}(\cdot),v^{\otimes(n-k-1)}\rangle\\
&\quad=n\langle\theta,v\rangle\sum_{k=0}^{n-1}{n-1 \choose k}\langle\rho_k,v^{\otimes k}\rangle  \langle P_{n-k-1}(\cdot),v^{\otimes(n-k-1)}\rangle=n\langle\theta,v\rangle \langle S_{n-1}(\cdot),v^{\otimes(n-1)}\rangle,
\end{align*}
which implies (S3). 

{\it Proof of\/}(S1)$\Leftrightarrow$(S6). The proof is similar to that of the equivalence of conditions (B4) and (B5) of Theorem~\ref{45678765g}, so we omit this proof. 
 
 We observe  that the $\rho_n$ in (S5) is of the form $\rho_n=n!\,A_n$. Indeed, setting $\omega=0$ in~\eqref{ldj84n}, we obtain
$\sum_{n=0}^\infty \frac1{n!}\langle S_n(0),v^{\otimes n}\rangle=A(v)=\sum_{n=0}^\infty \langle A_n,v^{\otimes n}\rangle$, hence $S_n(0)=n!\,A_n$ for $n\in\mathbb N_0$. On the other hand, we already deduced from \eqref{fgxtgxt} that  $S_n(0)=\rho_n$. 

Finally, the proof of (S2)$\Rightarrow$(S1), in particular, formula \eqref{vfsrar45w5y}, implies that the operator $T$ in (S2) satisfies $(\mathcal I^{-1}T)(v)=\tau(v)=A^{-1}(B^{\langle -1\rangle}(v))=A^{-1}(L(v))$.
 \end{proof}
 
 \begin{corollary} Let $S\in\mathbb S(V)$ and let $(S_n(\omega))_{n=0}^\infty$ be the corresponding Sheffer sequence. Then $S\in\mathbb B(V)$ if and only if $S_0=1$ and $S_n(0)=0$ for all $n\in\mathbb N$.
  \end{corollary}
  
  \begin{proof}  We have $S\in \mathbb B(V)$ if and only if $A(v)=1$. But $A(v)=\sum_{n=0}^\infty \frac1{n!}\langle S_n(0),v^{\otimes n}\rangle$, which implies the corollary.
  \end{proof}

  \begin{corollary}\label{dfgh765dcfg} Let $S=[S_{kn}]_{k,n\in\mathbb N_0}\in\mathbb S(V)$ 
  and let the corresponding Sheffer sequence $(S_n(\omega))_{n=0}^\infty$ have the generating function   
  \eqref{ldj84n}. Then we have $S_{0n}=n!\,A_n$ for $n\in\mathbb N_0$, $S_{nn}=A_0B_1^{\otimes n}$ for $n\in\mathbb N$, and 
   \begin{equation*}S_{kn}=\frac{n!}{k!}\sum_{m=0}^{n-k}\sum_{{l_1,\dots,l_k \geq1}\atop l_1+\dots+l_k=n-m}(A_m\odot B_{l_1}\odot \dots \odot B_{l_k}), \quad 1\le k<n.
   \end{equation*}
\end{corollary}

\begin{proof}
The corollary easily follows from formula \eqref{fgh561560}. 
\end{proof}

We will now discuss an extension of Corollary \ref{cxtsxtstst} to the case of Sheffer sequences. Our result below generalizes  \cite[Corollary 7.2]{FKLO}.

Assume $(S_n(\omega))_{n=0}^\infty$ is a Sheffer sequence with generating function \eqref{ldj84n} in which $A(0)=S_0=1$.  We define 
$$\alpha(v):=\log(A(v))=\sum_{n=1}^\infty\frac{(-1)^{n+1}}n\,(A(v)-1)^n.$$
(Note that $\alpha(0)=0$)
Then the generating function of the Sheffer sequence has the form 
\begin{equation}\label{fgh784gh74}\sum_{n=0}^\infty \frac1{n!}\langle S_n(\omega),v^{\otimes n}\rangle=\exp\big[\langle\omega,B(v)\rangle+\alpha(v)].\end{equation}

The proof of the following corollary is similar to the proof of Corollary~\ref{cxtsxtstst}, so we  omit it. 

\begin{corollary}\label{dxzdrzdszrsxddx} Let $(S_n(\omega))_{n=0}^\infty$ be a Sheffer sequence with generating function \eqref{fgh784gh74}. 
 Let $B(v)=\sum_{k=1}^\infty\frac 1{k!}\,\tilde B_kv^{\otimes k}$ and $\alpha(v)=\sum_{k=1}^\infty\frac 1{k!}\,\langle \tilde \alpha_k,v^{\otimes k}\rangle$.
Then, for any $\omega\in V^*$,  $n\in\mathbb N$, and $v_1,\dots,v_n\in V$, we have 
\begin{equation}
\langle S_n(\omega),v_1\odot\dots\odot v_n\rangle=\sum_{\pi\in\mathcal P(n)}\prod_{\mathcal A\in\pi}\Big(\big\langle\omega,\tilde B_{|\mathcal A|}\underset{i\in \mathcal A}\odot v_i\big\rangle+\big\langle \tilde\alpha_{|\mathcal A|},\underset{i\in \mathcal A}\odot v_i\big\rangle\Big).\label{cxrtstset55}
\end{equation}
In particular, for $\omega\in V^*$, $v\in V$, and $n\in\mathbb N$,
\begin{equation}
\label{dxtststs}
\langle S_n(\omega),v^{\otimes n}\rangle=\sum_{\pi\in\mathcal P(n)}\prod_{\mathcal A\in\pi}\big(\langle\omega,\tilde B_{|\mathcal A|}v^{\otimes|\mathcal A|}\rangle+\langle\tilde\alpha_{|\mathcal A|},v^{\otimes|\mathcal A|}\rangle\big).
\end{equation}
\end{corollary}

\begin{corollary}[Factorization property of Sheffer sequences]\label{stesw5twu5}
Let $(S_n(\omega))_{n=0}^\infty$ be a Sheffer sequence with generating function \eqref{fgh784gh74}.   Let $B(v)=\sum_{k=1}^\infty B_kv^{\otimes k}$ and  $\alpha(v)=\sum_{k=1}^\infty\langle  \alpha_k ,v^{\otimes k}\rangle$. Let vectors $v_1,\dots,v_K\in V$ be such that for any $l_1,\dots,l_K\in \mathbb N_0$ with $l_i,l_j\in \mathbb N$ for some $i\neq j$, we have
\begin{equation*}B_{l_1+\dots+l_K}(v_1^{\otimes l_1}\odot \cdots\odot v_K^{\otimes l_K})=0,\quad \langle 
\alpha_{l_1+\dots+l_K},v_1^{\otimes l_1}\odot \cdots\odot v_K^{\otimes l_K}\rangle=0.\end{equation*}
Then, for any $l_1,\dots,l_K\in \mathbb N_0$,
\begin{equation*}
\langle S_{l_1+\cdots+l_K}(\omega), v^{\otimes l_1}\odot \cdots\odot v^{\otimes l_K}\rangle =\langle S_{l_1}(\omega),v^{\otimes l_1}\rangle\cdots \langle S_{l_K}(\omega),v_K^{\otimes l_K}\rangle.\end{equation*}
\end{corollary}

\begin{proof} To simplify the notation, we will prove the corollary for two vectors $v$ and $w$ such that $B_{k+l}(v^{\otimes k}\odot w^{\otimes l})=0$ and $\langle \alpha_{k+l},v^{\otimes k}\odot w^{\otimes l}\rangle=0$ for all $k,l\in \mathbb N$. 

Let $s,t\in \mathbb F$. Then, for each $n\in \mathbb N$, $n\geq2$,
\begin{equation*}B_n(sv+tw)^{\otimes n}=\sum_{k=0}^n{n\choose k}s^kt^{n-k}B_n(v^{\otimes k}\odot w^{\otimes (n-k)})=s^nB_nv^{\otimes n}+t^nB_nw^{\otimes n},\end{equation*}
and similarly $\langle \alpha_n,(sv+tw)^{\otimes n}\rangle=s^n\langle\alpha_n,v^{\otimes n}\rangle+t^n\langle\alpha_n,w^{\otimes n}\rangle$. 
Therefore, by \eqref{fgh784gh74},
\begin{align}\label{453tg376d8}
\sum_{n=0}^\infty \frac 1{n!}\langle S_n(\omega),(sv+tw)^{\otimes n}\rangle&=\exp\big[\langle\omega,B(sv)\rangle+\alpha(sv)\big]\exp\big[\langle\omega,B(tw)\rangle+\alpha(tw)\big]\notag\\
&=\sum_{k=0}^\infty \frac {s^k}{k!}\langle S_k(\omega),v^{\otimes k}\rangle\sum_{l=0}^\infty \frac {t^l}{l!}\langle S_l(\omega),w^{\otimes l}\rangle\notag\\
&=\sum_{k=0}^\infty\sum_{l=0}^\infty \frac{s^kt^l}{k!\,l!}\langle S_k(\omega),v^{\otimes k}\rangle\langle S_l(\omega),w^{\otimes l}\rangle.\end{align}
On the other hand,
\begin{align}\label{geduye378d}
\sum_{n=0}^\infty \frac 1{n!}\langle S_n(\omega),(sv+tw)^{\otimes n}\rangle&=\sum_{n=0}^\infty\frac1{n!}\sum_{k=0}^n{n \choose k}s^kt^{n-k}\langle S_n(\omega),v^{\otimes k}\odot w^{\otimes (n-k)}\rangle\notag\\
& =\sum_{n=0}^\infty\sum_{k=0}^n\frac1{k!\,(n-k)!}s^kt^{n-k}\langle S_n(\omega),v^{\otimes k}\odot w^{\otimes (n-k)}\rangle\notag\\
&=\sum_{k=0}^\infty\sum_{l=0}^\infty\frac{s^kt^l}{k!\,l!}\langle S_{k+l}(\omega),v^{\otimes k}\odot w^{\otimes l}\rangle.
\end{align}
Note that both \eqref{453tg376d8} and \eqref{geduye378d} can be thought of as formal power series in $s$ and $t$. Hence, by \eqref{453tg376d8} and \eqref{geduye378d}, we get 
\begin{equation*}
\langle S_{k+l}(\omega),v^{\otimes k}\odot w^{\otimes l}\rangle=\langle S_k(\omega),v^{\otimes k}\rangle \langle S_l(\omega), w^{\otimes l}\rangle,\quad k,l\in \mathbb N.\qedhere\end{equation*}
\end{proof}

We finish this section with a brief discussion of Appell sequences. Let $S\in\mathbb S(V)$ and let $(S_n(\omega))_{n=0}^\infty$ be the corresponding Sheffer sequence. We say that $S$  is an {\it Appell operator} and $(S_n(\omega))_{n=0}^\infty$ is an {\it Appell sequence} if the generating function of $(S_n(\omega))_{n=0}^\infty$ is given by formula \eqref{ldj84n} in which $B(v)=v$.  We denote by $\mathbb A(V)$ the set of all Appell operators.

By Theorem~\ref{ctrst5u} (S3), a polynomial sequence $(S_n(\omega))_{n=0}^\infty$ is an Appel sequence if and only if $D$ is its lowering differential.
By Theorem~\ref{ctrst5u} (S5), a polynomial sequence $(S_n(\omega))_{n=0}^\infty$ is an Appel sequence if and only if there exists $\rho(v)=\sum_{k=0}^\infty \langle \rho_k, v^{\otimes n}\rangle\in \mathcal F_0(V)$ such that, for all $\omega \in V^*$  and $n \in \mathbb N$,
$S_n(\omega) =\sum_{k=0}^n{n \choose k}
\rho_k\odot \omega^{\otimes (n-k)}$.

\section{Recurrence formulas}\label{cxdzrzrzarwar}

 Let $(S_n(\omega))_{n=0}^\infty$ be a Sheffer sequence with generating function \eqref{fgh784gh74}, in which $B(v)=\sum_{k=1}^\infty B_kv^{\otimes k}$ and $\alpha(v)=\sum_{k=1}^\infty\langle \alpha_k,v^{\otimes k}\rangle$.   
 Recall that the lowering differential $Q$ corresponding to the Sheffer sequence $(S_n(\omega))_{n=0}^\infty$ is given by formula \eqref{vcftdtrdsy6ek} in which $L(v)=\sum_{k=1}^\infty L_kv^{\otimes k}=B^{\langle-1\rangle}(v)$.

 For each $\zeta\in V$, we define the 
 {\it raising  operator $R(\zeta)\in \mathcal L(\mathcal P  (V^*))$ corresponding to the Sheffer sequence $(S_n(\omega))_{n=0}^\infty$}  by requiring that 
 $$\big(R(\zeta)\langle S_n(\cdot),f_n\rangle\big)(\omega)=\langle S_{n+1}(\omega),f_n\odot \zeta\rangle,\quad f_n\in V^{\odot n},\ n\in\mathbb N_0.$$

We will now obtain  an explicit formula for the raising 
  operators. To this end, for each fixed $\zeta\in V$, we define formal tensor power series 
\begin{align}
\Phi_\zeta(v):&=B'(v;\zeta)\circ L(v)=B'(L(v);\zeta)\in\mathcal F(V,V),\notag\\
\Gamma_\zeta(v):&=\alpha'(v;\zeta)\circ L(v)=\alpha'(L(v);\zeta)\in\mathcal F(V,\mathbb F). 
\label{vcrdr6ts5a43aa}\end{align}

\begin{theorem}\label{dzerste7kui7} Under the above assumptions, we have, for each $\zeta\in V$, $p\in\mathcal P(V^*)$ and $\omega\in V^*$,
\begin{equation}\label{dtrd6t5swe5u4wu5}
\big(R(\zeta)p\big)(\omega)=
\big\langle\omega, (\Phi_\zeta(D)p)(\omega)\big\rangle+ (\Gamma_\zeta(D)p)(\omega). 
\end{equation}
\end{theorem}

\begin{remark} In the case where $(S_n(\omega))_{n=0}^\infty=(P_n(\omega))_{n=0}^\infty$ is a binomial sequence, we have $\alpha(v)=0$, hence formula~\eqref{dtrd6t5swe5u4wu5} simplifies as follows: $ \big(R(\zeta)p\big)(\omega)=
\big\langle\omega, (\Phi_\zeta(D)p)(\omega)\big\rangle$. \end{remark}

\begin{remark}
In view of  \eqref{crstese5ws5},  \eqref{vcfxdtrstea} and \eqref{vcrdr6ts5a43aa}, formula \eqref{dtrd6t5swe5u4wu5} can be written as follows:
\begin{align}
&(R(\zeta)p)(\omega)=\big(\langle\omega, B_1\zeta\rangle+\langle\alpha_1,\zeta\rangle\big) p(\omega)\notag\\
&\quad+\sum_{k=1}^\infty \sum_{n=1}^k(n+1) \sum_{i_1,\dots,i_n \geq 1 \atop i_1+\dots+i_n=k} 
\big\langle \omega B_{n+1}+\alpha_{n+1},\zeta\odot\big[(L_{i_1}\odot\dots\odot L_{i_n}) D^{k}p(\omega)\big]\big\rangle.\label{vcryds6uei}
\end{align}
\end{remark}

\begin{remark} \label{cftsts6e}
Consider the one-dimensional case, $V=\mathbb F$. Denote by $\Phi(t)$ and $\Gamma(t)$ the formal power series $\Phi_\zeta(t)$ and $\Gamma_\zeta(t)$ for $\zeta=1$. (Here $t$ is 
the variable from $\mathbb F$). Since $L(t)$ is the compositional inverse of $B(t)$ and $\alpha(t)=\log(A(t))$, we obtain:
\begin{align}
\Phi(t)&=B'(t)\circ L(t)=\frac 1{L'(t)},\notag\\
 \Gamma(t)&=\alpha'(t)\circ L(t)=\frac{A'(t)}{A(t)}\circ L(t)=\frac{C'(t)}{C(t)}\cdot \frac 1{L'(t)},\label{dydr758i58o58o}
 \end{align} 
where $C(t):=A(t)\circ L(t)$. Also, denoting $R:=R(1)$, we obtain from \eqref{dtrd6t5swe5u4wu5}: 
\begin{equation}\label{xdsre5sw5w5}
(Rp)(x)=x(\Phi(D)p)(x)+(\Gamma(D)p)(x).
\end{equation}
 In particular, in the case of a binomial sequence, formula \eqref{xdsre5sw5w5} becomes $(Rp)(x)=x(\Phi(D)p)(x)$, which is  called {\it Rodrigues formula\/} in \cite[Section~4, Theorem~4~(4)]{RotaKahanerOdlyzko}. In the case of a general Sheffer sequence over $\mathbb F$, formulas \eqref{dydr758i58o58o}, \eqref{xdsre5sw5w5} are shown in \cite[Theorem~3.7.1]{Roman}. Formula~\eqref{xdsre5sw5w5} is often called the {\it recurrence formula}.

\end{remark}

\begin{proof}[Proof of Theorem~\ref{dzerste7kui7}] We divide the proof into several 
steps.

{\it Step 1}. For each $k\in\mathbb N$, let us fix  $T_k\in(V^{\odot k})^*$. For $k,n\in\mathbb N$ with $n\ge k$ and any $v_1,\dots,v_n\in V$, formula \eqref{ge4ew34wq43w} with $Q=D$ implies
$$ \frac1{k!}\,\big(\langle T_k ,D^{k}\rangle\langle\cdot ^{\otimes n},v_1\odot\dots\odot v_n\rangle\big)(\omega)=\sum_{{\mathcal A\subset\{1,\dots,n\}}\atop{|\mathcal A|=k}}\big\langle T_k ,\underset{i\in \mathcal A}\odot v_i\big\rangle\big\langle\omega^{\otimes(n-k)},\underset{j\in \mathcal A^c}\odot v_j\big\rangle,$$
where $\mathcal A^c=\{1,\dots,n\}\setminus\mathcal A$. Hence, for any constants $c_1,\dots,c_n\in\mathbb F$, we obtain
$$ \frac1{k!}\,\bigg(\langle T_k, D^{k}\rangle\prod_{i=1}^n(\langle\cdot,v_i\rangle+c_i)\bigg)(\omega)=\sum_{{\mathcal A\subset\{1,\dots,n\}}\atop{|\mathcal A|=k}}\big\langle T_k, \underset{i\in \mathcal A}\odot v_i\big\rangle\prod_{j\in\mathcal A^c}(\langle\omega,v_j\rangle+c_j).$$
Therefore,
\begin{equation}\label{dxters5ewu}
\bigg(\sum_{k=1}^\infty \frac 1{k!}\,\langle T_k, D^{k}\rangle\prod_{i=1}^n(\langle\cdot,v_i\rangle+c_i)\bigg)(\omega) =\sum_{{\mathcal A\subset\{1,\dots,n\}}\atop{\mathcal A\ne\varnothing}}\big\langle T_{|\mathcal A|}, \underset{i\in \mathcal A}\odot v_i\big\rangle\prod_{j\in\mathcal A^c}(\langle\omega,v_j\rangle+c_j).\end{equation}

{\it Step 2}. Let $\pi=\{\mathcal A_1,\dots,\mathcal A_k\}\in\mathcal P(n)$. We write $\rho\subset\pi$, $\rho\ne\varnothing$ if $\rho=\{\mathcal A_{i_1},\dots,\mathcal A_{i_m}\}$, where $\{i_1,\dots,i_m\}$ is a nonempty subset of $\{1,\dots,k\}$. By \eqref{dxters5ewu}, we have for each $\pi\in\mathcal P(n)$,
\begin{align}
&\bigg(\sum_{k=1}^\infty \frac 1{k!}\,\langle T_k, D^{k}\rangle \prod_{\mathcal A\in\pi}\Big(
\big\langle \cdot,\tilde B_{|\mathcal A|}\underset{i\in \mathcal A}\odot v_i\big\rangle+\big\langle \tilde\alpha_{|\mathcal A|},\underset{i\in \mathcal A}\odot v_i\big\rangle\Big)\bigg)(\omega)\notag\\
&\quad=\sum_{\rho\subset\pi\atop \rho\ne\varnothing}\Big\langle T_{|\rho|},
\underset{\mathcal A\in\rho}\odot\Big(\tilde B_{|\mathcal A|} \underset{i\in\mathcal A}\odot v_i\Big)\Big\rangle
\prod_{\mathcal B\in\pi\atop\mathcal B\not\in \rho}
\bigg(
\big\langle \omega,\tilde B_{|\mathcal B|}\underset{j\in \mathcal B}\odot v_j\big\rangle+\big\langle \tilde\alpha_{|\mathcal B|},\underset{j\in \mathcal B}\odot v_j\big\rangle\bigg).\label{xersa5y}
\end{align}
By \eqref{cxrtstset55} and \eqref{xersa5y}, we have, for  $v_1,\dots,v_n\in V$,
\begin{align}
&\bigg(\sum_{k=1}^\infty \frac 1{k!}\,\langle T_k, D^{k}\rangle\langle S_n(\cdot),v_1\odot\dots\odot v_n\rangle\bigg)(\omega)\notag\\ 
&\quad=\sum_{\pi\in\mathcal P(n)}\sum_{\mathcal A\in\pi}
\bigg(\sum_{\rho\in\mathcal P(\mathcal A)}\Big\langle T_{|\rho|},\underset{\mathcal B\in\rho}{\odot}\big(\tilde B_{|\mathcal B|}\underset{i\in\mathcal B}\odot v_i\big)\Big\rangle \bigg)
\prod_{{\mathcal C\in\pi}\atop{\mathcal C\ne\mathcal A}}\Big(\langle\omega,\tilde B_{|\mathcal C|}\underset{j\in\mathcal C}\odot v_j \rangle+\langle\tilde\alpha_{|\mathcal C|},
\underset{j\in\mathcal C}\odot v_j\rangle\Big).\label{vytdyd}
\end{align}
(In  formula \eqref{vytdyd}, if $\pi$ contains a single set $\mathcal A=\{1,\dots,n\}$, the product over $\mathcal C\in\pi$ with $\mathcal C\ne\mathcal A$ is supposed to be equal to 1.)

{\it Step 3}. Denote by $(P_n(\omega))_{n=0}^\infty$ the binomial sequence corresponding to the Sheffer sequence $(S_n(\omega))_{n=0}^\infty$, i.e., $(P_n(\omega))_{n=0}^\infty$ has generating function \eqref{trse5sw5w}. Then, by formula~\eqref{vytdyd}, 
\begin{align}
&\bigg(\sum_{k=1}^\infty \frac 1{k!}\,\langle T_k, D^{k}\rangle\langle P_n(\cdot),v_1\odot\dots\odot v_n\rangle\bigg)(\omega)\notag\\ 
&\quad=\sum_{\pi\in\mathcal P(n)}\sum_{\mathcal A\in\pi}
\bigg(\sum_{\rho\in\mathcal P(\mathcal A)}\Big\langle T_{|\rho|},\underset{\mathcal B\in\rho}{\odot}\big(\tilde B_{|\mathcal B|}\underset{i\in\mathcal B}\odot v_i\big)\Big\rangle \bigg)
\prod_{{\mathcal C\in\pi}\atop{\mathcal C\ne\mathcal A}}\langle\omega,\tilde B_{|\mathcal C|}\underset{j\in\mathcal C}\odot v_j \rangle.\notag
\end{align}
Therefore,
\begin{equation}\label{fdre654w6}
\bigg(\sum_{k=1}^\infty \frac 1{k!}\,\langle T_k, D^{k}\rangle\langle P_n(\cdot),v_1\odot\dots\odot v_n\rangle\bigg)(0)= \sum_{\rho\in\mathcal P(n)}\Big\langle T_{|\rho|},\underset{\mathcal B\in\rho}{\odot}\big(\tilde B_{|\mathcal B|}\underset{i\in\mathcal B}\odot v_i\big)\Big\rangle .
\end{equation}

{\it Step 4}. 
For each $\zeta\in V$ and $n\in\mathbb N$, we define $\widehat B_{n}(\zeta)\in\mathcal L(V^{\odot n},V)$ and $\widehat\alpha_{n}(\zeta)\in (V^{\odot n})^*$ by 
\begin{equation}\label{dy6de6ir5}
\widehat B_{n}(\zeta)f_n:=\tilde B_{n+1}(\zeta\odot f_n),\quad \langle\widehat\alpha_{n}(\lambda),f_n\rangle:=\langle\tilde\alpha_{n+1},\zeta\odot f_n\rangle,\quad f_n\in V^{\odot n}.
\end{equation}
Then, by formula \eqref{cxrtstset55}, we have, for $\zeta,v\in V$ and $n\in\mathbb N$,
\begin{align}
&\langle S_{n+1}(\omega),\zeta\odot v^{\otimes n}\rangle=\big(\langle\omega, B_1\zeta\rangle+\langle \alpha_1,\zeta\rangle\big)\langle S_n(\omega),v^{\otimes n}\rangle\notag\\
&\quad+\sum _{\pi\in\mathcal P(n)}\sum_{\mathcal A\in\pi} \big(\langle\omega,\widehat B_{|\mathcal A|}(\zeta)  v^{\otimes|\mathcal A|}\rangle+\langle \widehat\alpha_{|\mathcal A|}(\zeta),v^{\otimes|\mathcal A|}\rangle\big) 
\prod_{{\mathcal B\in\pi}\atop{\mathcal B\ne\mathcal A}}\big(\langle\omega,\tilde B_{|\mathcal B|}v^{\otimes|\mathcal B|}\rangle+\langle\tilde\alpha_{|\mathcal B|},v^{\otimes|\mathcal B|}\rangle\big).\label{cdtre6u}
\end{align}
Note also that
\begin{equation}
\langle S_{1}(\omega),\zeta\rangle= \langle\omega, B_1\zeta\rangle+\langle \alpha_1,\zeta\rangle.\label{vdtrs6ue}
\end{equation}

 {\it Step 5}. Fix $\zeta\in V$. Assume that $\Theta_k(\zeta)\in \mathcal L(V^{\odot k},V)$ and $\Xi_k(\zeta)\in(V^{\odot k})^*$
($k\in\mathbb N$) are 
 such that, for each  $v\in V$ and $n\in\mathbb N$,
\begin{align}
&\sum_{\rho\in\mathcal P(n)}  \Theta_{|\rho|}(\zeta)\bigg(\underset{\mathcal B\in\rho}{\odot}\big(\tilde B_{|\mathcal B|}v^{\otimes |\mathcal B|}\big)\bigg)=\widehat B_{n}(\zeta)v^{\otimes n} ,\label{vfyxreasry5ay45}\\
&\sum_{\rho\in\mathcal P(n)} \Big\langle \Xi_{|\rho|}(\zeta), \underset{\mathcal B\in\rho}{\odot}\big(\tilde B_{|\mathcal B|}v^{\otimes |\mathcal B|}\big)\Big\rangle =\langle \widehat\alpha_{n}(\zeta),v^{\otimes n}\big\rangle.
\label{bvyuf7}\end{align}
Then, by \eqref{vytdyd}, \eqref{vfyxreasry5ay45} and \eqref{bvyuf7}, we obtain, for $\omega\in V^*$ and $v\in V$,
\begin{align}
&\bigg(\sum_{k=1}^\infty \frac 1{k!}\,\langle \omega \Theta_k(\zeta)+\Xi_k(\zeta), D^{k}\rangle
\langle S_n(\cdot),v^{\otimes n}\rangle\bigg)(\omega)\notag\\ 
&\quad=\sum_{\pi\in\mathcal P(n)}\sum_{\mathcal A\in\pi}
\big(\langle\omega,\widehat B_{|\mathcal A|}v^{\otimes|\mathcal A|}\rangle+\langle\widehat \alpha_{|\mathcal A|},v^{\otimes |\mathcal A|}\rangle \big)
\prod_{{\mathcal B\in\pi}\atop{\mathcal B\ne\mathcal A}}\Big(\langle\omega,\tilde B_{|\mathcal B|}v^{\otimes|\mathcal B|} \rangle+\langle\tilde\alpha_{|\mathcal B|},
v^{\otimes|\mathcal B|}\rangle\Big).\label{vcfxtsxte5su5}
\end{align}

By \eqref{cdtre6u} and \eqref{vcfxtsxte5su5}, we have, for each $p\in\mathcal P(V^*)$, $\omega\in V^*$ and $v\in V$,
\begin{align}
(R(\zeta)p)(\omega)&=\big(\langle\omega, B_1\zeta\rangle+\langle \alpha_1,\zeta\rangle\big)p(\omega)+\bigg\langle\omega, \bigg(\sum_{k=1}^\infty \frac 1{k!}\,\Theta_k(\zeta)D^kp\bigg)(\omega)\bigg\rangle\notag\\
&\quad+\bigg(\sum_{k=1}^\infty \frac 1{k!}\,\langle \Xi_k(\zeta), D^k\rangle p\bigg)(\omega).\label{vgdctrs5ea43y7}
\end{align}

{\it Step 6}. In view of \eqref{fdre654w6}, conditions \eqref{vfyxreasry5ay45},  \eqref{bvyuf7}   can be reformulated as follows: for each 
$\zeta\in V$, $v\in V$ and $n\in\mathbb N$, 
\begin{align}
&\bigg(\sum_{k=1}^\infty \frac 1{k!}\,\Theta_k(\zeta) D^{k} \langle P_n(\cdot),v^{\otimes n}\rangle\bigg)(0)=
\widehat B_{n}(\zeta)v^{\otimes n},\label{vdrste5s5sw5}\\
&\bigg(\sum_{k=1}^\infty \frac 1{k!}\,\langle \Xi_k(\zeta), D^{k} \rangle\langle P_n(\cdot),v^{\otimes n}\rangle\bigg)(0)=\langle \widehat\alpha_{n}(\zeta),v^{\otimes n}\rangle.
\label{dzrwaw}\end{align}

By \eqref{ge4ew34wq43w} and Lemma \ref{gfyrdr6ew64ui}, we have, for any $T_k\in(V^{\odot k})^*$ ($k\in\mathbb N$), $v\in V$ and $n\in\mathbb N$,
\begin{equation*}
\bigg(\sum_{k=1}^\infty \frac 1{k!}\,\langle T_k, Q^{k}\rangle\langle P_n(\cdot),v^{\otimes n}\rangle\bigg)(0)=\langle T_n,v^{\odot n}\rangle.
\end{equation*}
Therefore,
\begin{align*}
\bigg(\sum_{k=1}^\infty \frac 1{k!}\,\widehat B_{k}(\zeta)Q^{k}\langle P_n(\cdot),v^{\otimes n}\rangle\bigg)(0)&=\widehat B_{n}(\zeta)v^{\odot n},\\
\bigg(\sum_{k=1}^\infty \frac 1{k!}\,\langle  \widehat\alpha_{k}(\zeta), Q^{k}\rangle\langle P_n(\cdot),v^{\otimes n}\rangle\bigg)(0)&=\langle  \widehat\alpha_{n}(\zeta),v^{\odot n}\rangle
\end{align*}
Hence, conditions \eqref{vdrste5s5sw5}, \eqref{dzrwaw} are satisfied provided the following equalities hold true:
\begin{align}
\sum_{k=1}^\infty \frac 1{k!}\,  \Theta_k(\zeta) D^{k} &=\sum_{k=1}^\infty \frac 1{k!}\,\widehat B_{k}(\zeta) Q^{k}\label{cdaewa4yw6u7j8},\\
\sum_{k=1}^\infty \frac 1{k!}\,  \langle \Xi_k(\zeta), D^{k}\rangle &=\sum_{k=1}^\infty \frac 1{k!}\,\langle \widehat \alpha_{k}(\zeta), Q^{k}\rangle.
\label{vyjdk7r8}\end{align}

In view of Corollaries~\ref{567ty7} (i) and \ref{xesa4aq46q327y} (ii) with $Q=D$, conditions \eqref{cdaewa4yw6u7j8}, \eqref{vyjdk7r8} can be written as follows:
\begin{align}
\sum_{k=1}^\infty \frac 1{k!}\,  \Theta_k(\zeta) v^{\otimes k}&=\sum_{k=1}^\infty \frac 1{k!}\,\mathcal J\big(\widehat B_{k}(\zeta) Q^{k}\big)(v),\label{vcyd6i5}\\
\sum_{k=1}^\infty \frac 1{k!}\, \langle \Xi_k(\zeta), v^{\otimes k}\rangle&=\sum_{k=1}^\infty \frac 1{k!}\,\mathcal I\big(\langle \widehat \alpha_{k}(\zeta), Q^{k}\rangle\big)(v).\label{dsthe}
\end{align}

We easily conclude from \eqref{ctxsts} that 
\begin{align}
&\sum_{k=1}^\infty \frac 1{k!}\,\mathcal J\big(\widehat B_{k}(\zeta) Q^{k}\big)(v)=
\sum_{k=1}^\infty \frac 1{k!}\,\widehat B_{k}(\zeta) L(v)^{\otimes k}
=\sum_{k=1}^\infty \frac 1{k!}\,\tilde B_{k+1}(\zeta\odot L(v)^{\otimes k})\notag\\
&\quad= \sum_{k=2}^\infty\frac{k}{k!}\,\tilde B_k(\zeta\odot L(v)^{\otimes(k-1)})=\breve B'(v;\zeta)\circ L(v). \label{vftr6e645}
\end{align}
Here, we used the formal tensor power series 
$$\breve B(v):=\sum_{k=2}^\infty \frac1{k!}\tilde B_kv^{\otimes k}=B(v)-B_1v.$$ 
Similarly,
\begin{equation}
\sum_{k=1}^\infty \frac 1{k!}\,\mathcal I\big(\widehat \alpha_{k}(\zeta) Q^{k}\big)(v)=\breve \alpha'(v;\zeta)\circ L(v),
\label{ryd6ue5i7}
\end{equation}
where 
$$\breve \alpha (v):=\sum_{k=2}^\infty \frac1{k!}\langle \tilde \alpha _k,v^{\otimes k}\rangle=\alpha(v)-\langle \alpha_1,v\rangle.$$ 

Thus, conditions \eqref{vcyd6i5}, \eqref{dsthe} are satisfied if we choose $\Theta_k(\zeta)$ and $\Xi_k(\zeta)$ as follows:
\begin{equation}\label{vcyd6uei6}
\sum_{k=1}^\infty \frac 1{k!}\,  \Theta_k(\zeta) v^{\otimes k}=\breve B'(v;\zeta)\circ L(v),\quad \sum_{k=1}^\infty \frac 1{k!}\,  \langle \Xi_k(\zeta), v^{\otimes k}\rangle=\breve\alpha'(v;\zeta)\circ L(v).\end{equation}

 Denote $\breve\Phi_\zeta(v):=\breve B'(v;\zeta)\circ L(v)$ and $\breve\Gamma_\zeta(v):=\breve \alpha'(v;\zeta)\circ L(v)$. Then, by \eqref{vgdctrs5ea43y7} and \eqref{vcyd6uei6}, we obtain, for 
 $p\in\mathcal P(V^*)$, $\zeta\in V$ and $\omega\in V^*$,
\begin{align*}
(R(\zeta)p)(\omega)&=\big(\langle\omega, B_1\zeta\rangle
+\langle\alpha_1,\zeta\rangle\big)p(\omega)+\langle\omega,(\breve\Phi_\zeta(D)p)(\omega)\rangle+(\breve\Gamma_\zeta(D)p)(\omega)\\
&=\langle\omega,(\Phi_\zeta(D)p)(\omega)\rangle+(\Gamma_\zeta(D)p)(\omega).\qedhere
\end{align*}
\end{proof}

 Let $(S_n(\omega))_{n=0}^\infty$ be a Sheffer sequence.  Our next aim is to represent the polynomial $\langle\omega,\zeta\rangle \langle S_n(\omega),f_n\rangle$ through the Sheffer polynomials of degrees $0,1,\dots,n+1$.  
To simplify the notation, we will assume that the Sheffer sequence is monic. The interested reader may easily extend our result to the case of a general Sheffer sequence.

 Let $(R(\zeta))_{\zeta\in V}$ be the family of the raising operators for the monic Sheffer sequence $(S_n(\omega))_{n=0}^\infty$\,. We define the {\it raising differential} $R\in\mathcal L\big(\mathcal P(V^*,V),\mathcal P(V^*)\big)$ that satisfies for each $\zeta\in V$, $\omega\in V^*$ and $p\in\mathcal P(V^*)$:
\begin{equation}\label{vcydyrked7ir5o7}
R(\zeta p(\cdot))(\omega)=(R(\zeta)p)(\omega).\end{equation}
By \eqref{dzsrea4ywu565468l} and \eqref{vcydyrked7ir5o7}, we have, for each $G_k\in\mathcal L(V^{\odot k},V)$ ($k\in\mathbb N$), 
 \begin{align*}
 \big(RG_kQ^k\langle S_n(\cdot),v^{\otimes n}\rangle\big)(\omega)&=(n)_kR(G_kv^{\otimes k})\langle S_{n-k}(\omega),v^{\otimes(n-k)}\rangle\big)\\
  &=(n)_k\langle S_{n-k+1}(\omega),v^{\otimes(n-k)}\odot(G_kv^{\otimes k})\rangle.
  \end{align*}

 For each $\zeta\in V$, we define
\begin{align}
\Psi_\zeta(v):&=L'(v;\zeta)\circ B(v)=L'(B(v);\zeta)\in\mathcal F(V,V),\notag\\
\Delta_\zeta(v):&=-\alpha'(v;\Psi_\zeta(v))\in\mathcal F(V,\mathbb F). 
\label{gfr6e54w38ffdrd}\end{align}
 We denote $\Psi_\zeta(v)=\sum_{k=0}^\infty \Psi_{\zeta,k}v^{\otimes k}$ and  $\Delta_\zeta(v)=\sum_{k=0}^\infty\langle \Delta_{\zeta,k},v^{\otimes k}\rangle$. Note that $\Psi_{\zeta,0}=\zeta$ and $\Delta_{\zeta,0}=-\langle\alpha_1,\zeta\rangle$.

For each $\zeta\in V$, we define the linear operator $M(\zeta)\in\mathcal L(\mathcal P(V^*))$ by $(M(\zeta)p)(\omega):=\langle\omega,\zeta\rangle p(\omega)$ for each $p\in\mathcal P(V^*)$.

 \begin{theorem}\label{xts5yw5yww45uwe} Let $(S_n(\omega))_{n=0}^\infty$ be a monic Sheffer sequence, and let $Q$ be its lowering differential. 
  Then, for each $\zeta\in V$,
 \begin{equation}\label{xresraqw5ywu5}
 M(\zeta)=R \Psi_\zeta(Q)+\Delta_\zeta(Q).
 \end{equation}
 Equivalently, for each $n\in\mathbb N_0$ and $v\in V$,
 \begin{align}
 &\langle\omega,\zeta\rangle\langle S_n(\omega),v^{\otimes n}\rangle=\langle S_{n+1}(\omega),v^{\otimes n}\odot\zeta\rangle\notag\\
&
 +\sum_{m=1}^{n}\Big\langle S_m(\omega), (n)_{n-m+1}v^{\otimes(m-1)}\odot (\Psi_{\zeta,n-m+1}v^{\otimes(n-m+1)})\notag\\
 &\qquad\quad+(n)_{n-m}v^{\otimes m}\langle\Delta_{\zeta,n-m},v^{\otimes(n-m)}\rangle\Big\rangle+n!\langle\Delta_{\zeta,n},v^{\otimes n}\rangle.\label{cxesa5yw7ypui8uytr}
 \end{align}
 \end{theorem}

 \begin{remark}\label{cftdsy6rue6i}
  We advise the reader to compare formula \eqref{cxesa5yw7ypui8uytr} with the recurrence formula obtained in  \cite[Theorem~2.3.4]{Watanabe2}.
  \end{remark}
  
  \begin{remark}Assume that, for each $\zeta \in V$, $\Psi_{\zeta,k}=0$ for $k\ge3$ and $\Delta_{\zeta,k}=0$ for $k\ge2$. Then, formula \eqref{cxesa5yw7ypui8uytr} has the following three-diagonal form:
  \begin{align}
  &\langle\omega,\zeta\rangle\langle S_n(\omega),v^{\otimes n}\rangle=\langle S_{n+1}(\omega),v^{\otimes n}\odot\zeta\rangle
  +\big\langle S_n(\omega), nv^{\otimes(n-1)}\odot(\Psi_{\zeta,1}v)+v^{\otimes n}\Delta_{\zeta,0}\big\rangle\notag\\
  &\quad+\big\langle S_{n-1}(\omega), n(n-1)v^{\otimes(n-2)}\odot(\Psi_{\zeta,1}v^{\otimes 2})+nv^{\otimes (n-1)}\langle\Delta_{\zeta,1}v\rangle\big\rangle.\notag
    \end{align}
 Such Sheffer sequences are important for probability, infinite-dimensional analysis and mathematical physics,  see e.g.\ \cite{BK, Lytvynov2003,Jansen1,Hida,ItoKubo}. 
  \end{remark}

  \begin{remark}
  Similarly to Remark~\ref{cftsts6e}, formula \eqref{gfr6e54w38ffdrd} takes the following form  in the case $V=\mathbb F$:
  $\Psi(t)=\frac1{B'(t)}$, $\Delta(t)=-\frac{\alpha'(t)}{B'(t)}$.
  \end{remark}
  
  \begin{proof}[Proof of Theorem~\ref{xts5yw5yww45uwe}]
 Substituting $L(v)$ into the formal tensor power series in $v$ on the left- and right-hand sides of formula 
 \eqref{fgh784gh74} gives: 
  \begin{equation}\label{fre7i4e7io4}
  \sum_{n=0}^\infty \frac1{n!}\langle S_n(\omega),L(v)^{\otimes n}\rangle=\exp\big[\langle\omega,v\rangle+\alpha(L(v))\big].
  \end{equation}
Let $\zeta\in V$.  We note that
$$\frac{d}{dt}\Big|_{t=0}\alpha(L(v+t\zeta))=\sum_{n=1}^\infty n\langle\alpha_n, L(v)^{\otimes(n-1)}\odot L'(v,\zeta)\rangle=\alpha'(L(v),L'(v,\zeta)). $$
Hence, the differentiation of the formal tensor power series in $v$ on the left- and right-hand sides of formula \eqref{fre7i4e7io4} in direction $\zeta$ gives:
\begin{align}
  &\sum_{n=1}^\infty \frac{1}{(n-1)!}\langle S_n(\omega),L(v)^{\otimes (n-1)}\odot L'(v,\zeta)\rangle\notag\\
  &\quad =\exp\big[\langle\omega,v\rangle+\alpha(L(v))\big]\big(\langle\omega,\zeta\rangle+\alpha'(L(v);L'(v,\zeta))\big).
 \label{yutfr65ue65e58} \end{align}
 Substituting $B(v)$ into the formal tensor power series in $v$ on the left- and right-hand sides of formula 
 \eqref{yutfr65ue65e58} gives: 
\begin{align}
  &\sum_{n=1}^\infty \frac1{(n-1)!}\langle S_n(\omega),v^{\otimes (n-1)}\odot L'(B(v),\zeta)\rangle\notag\\
  &\quad =\exp\big[\langle\omega,B(v)\rangle+\alpha(v)]\big(\langle\omega,\zeta\rangle+\alpha'(v;L'(B(v),\zeta))\big).
 \label{tre46e34344} \end{align}
 In view of \eqref{fgh784gh74} and \eqref{gfr6e54w38ffdrd}, formula \eqref{tre46e34344}  can be written as follows:
 \begin{equation*}
  \sum_{n=0}^\infty \frac1{n!}\langle S_{n+1}(\omega),v^{\otimes n }\odot\Psi_\zeta(v)\rangle
    =\sum_{n=0}^\infty \frac1{n!}\langle S_n(\omega),v^{\otimes n}\rangle  
  \big(\langle\omega,\zeta\rangle-\Delta_{\zeta}(v)\big),
\end{equation*}
or equivalently, 
 \begin{align}
 &\sum_{n=0}^\infty \frac1{n!}\langle S_n(\omega),v^{\otimes n}\rangle\langle\omega,\zeta\rangle
 =\sum_{n=0}^\infty \sum_{m=1}^{n+1}\frac1{(m-1)!}\langle S_m(\omega),(\mathbf 1_{m-1}\odot\Psi_{\zeta,n-m+1})v^{\otimes n}\rangle \notag\\
 &\quad+\sum_{n=0}^\infty \sum_{m=0}^{n}\frac1{m!}\langle S_m(\omega),(\mathbf 1_{m}\odot\Delta_{\zeta,n-m})v^{\otimes n}\rangle.\notag
 \end{align}
 Therefore,  for each $n\in\mathbb N_0$,
 \begin{align}
&\langle S_n(\omega),v^{\otimes n}\rangle\langle\omega,\zeta\rangle=\langle S_{n+1}(\omega),v^{\otimes n}\odot\zeta\rangle\notag\\
&\quad+
\sum_{m=1}^{n}\frac{n!}{(m-1)!}\langle S_m(\omega),v^{\otimes(m-1)}\odot(\Psi_{\zeta,n-m+1}v^{\otimes(n-m+1)})\rangle\notag\\
&\quad + \sum_{m=0}^{n}\frac{n!}{m!}\langle S_m(\omega),v^{\otimes m}\rangle \langle \Delta_{\zeta,n-m},v^{\otimes(n-m)}\rangle,\notag
 \end{align} 
 which implies formula \eqref{cxesa5yw7ypui8uytr}, hence also formula \eqref{xresraqw5ywu5}. 
  \end{proof}

\section{The Sheffer group} \label{cfdrseepioioi}

\begin{theorem} \label{cxdsresay5e}
{\rm (i)} The set  $\mathbb S(V)$ of all Sheffer operators, equipped with the product of linear operators, is a group, which will be called the Sheffer group (over $V$). 
The set of Appell operators , $\mathbb A(V)$, is a normal subgroup of $\mathbb S(V)$, and the set of umbral operators, $\mathbb B(V)$, is a   subgroup of $\mathbb S(V)$. Furthermore, the group $\mathbb S(V)$ is the semidirect product of~$\mathbb A(V)$ and~$\mathbb B(V)$, i.e., $\mathbb S(V)=\mathbb A(V)\rtimes \mathbb B(V)$.

{\rm (ii)} We define a bijective map $\mathfrak R:\mathbb S(V) \to \mathcal R(V)$ as follows: for each Sheffer operator~$S$ whose Sheffer sequence $(S_n(\omega))_{n=0}^\infty$ has generating function \eqref{ldj84n}, we set $\mathfrak R(S):=(A(v),B(v))$. The map $\mathfrak  R$ is a group isomorphism between the Sheffer group $\mathbb S(V)$ and the Riordan group $\mathcal R(V)$.
\end{theorem}

\begin{remark}
It follows from Theorem~\ref{cxdsresay5e} that the restriction of the map $\mathfrak R$ to $\mathbb A(V)$ provides a group isomorphism between $\mathbb A(V)$ and $\mathcal F_0(V)$, and the restriction of the map $\mathfrak R$ to $\mathbb B(V)$ provides a group isomorphism between $\mathbb B(V)$ and $\mathcal F_1(V)$. (Here $\mathcal F_0(V)$ and $\mathcal F_1(V)$ are considered as subgroups of $\mathcal R(V)$.)
\end{remark}

\begin{proof}[Proof of Theorem \ref{cxdsresay5e}] It follows from the definition of a Sheffer operator that the map  $\mathfrak R:\mathbb S(V) \to \mathcal R(V)$ is indeed bijective, and furthermore $\mathfrak R\,\mathbb A(V)=\mathcal F_0(V)$ and $\mathfrak R\,\mathbb B(V)=\mathcal F_1(V)$. 

For $i=1,2$, let $S^{(i)}=[S_{kn}^{(i)}]_{k,n\in\mathbb N_0}\in \mathbb S(V)$ and $\mathfrak R(S^{(i)})=(A^{(i)}(v),B^{(i)}(v))$. Let $S=[S_{kn}]_{k,n\in\mathbb N_0}:=S^{(1)}S^{(2)}$, so that 
$S_{kn}=\sum_{i=k}^nS_{ki}^{(1)}S_{in}^{(2)}$. Then, for each $k\in \mathbb N$, we have, by Theorem~\ref{ctrst5u} (S6),
\begin{align*}
&\sum_{n=k}^\infty \frac{k!}{n!}S_{kn}v^{\otimes n}=\sum_{n=k}^\infty \frac{k!}{n!}\sum_{i=k}^nS_{ki}^{(1)}S_{in}^{(2)}v^{\otimes n}=\sum_{i=k}^\infty \sum_{n=i}^\infty\frac{k!}{n!}\,S_{ki}^{(1)}S_{in}^{(2)}v^{\otimes n}\\
&\quad=\sum_{i=k}^\infty \frac{k!}{i!}\,S_{ki}^{(1)}\sum_{n=i}^\infty \frac{i!}{n!}S_{in}^{(2)}v^{\otimes n}=\sum_{i=k}^\infty \frac{k!}{i!}\,S_{ki}^{(1)}B^{(2)}(v)^{\otimes i}A^{(2)}(v)\\
&\quad =\bigg(\sum_{i=k}^\infty \frac{k!}{i!}\,S_{ki}^{(1)}B^{(2)}(v)^{\otimes i}\bigg)A^{(2)}(v)=B^{(1)}(B^{(2)}(v))^{\otimes k}A^{(1)}(B^{(2)}(v))A^{(2)}(v)).\end{align*}
Therefore, $S\in\mathbb S(V)$ and $\mathfrak R(S)=(A(v),B(v))$, where
$A(v)=A^{(1)}(B^{(2)}(v))A^{(2)}(v)$ and $B(v)=B^{(1)}(v)\circ B^{(2)}(v)$. Hence, by \eqref{bvfrsxrea4qw}, $\mathfrak R(S^{(1)}S^{(2)})=\mathfrak R(S^{(1)})\ast\mathfrak R(S^{(2)})$. 
 Now, both statements of the theorem follow immediately from the fact that $\mathcal R(V)$ is a group and formula~\eqref{cxdsresresuy} holds.
\end{proof}

\section{Lifting of Sheffer sequences over $\mathbb F$  to Sheffer\\ sequences over $V$}\label{gxzrfsarh}

\subsection{General theory}\label{87ygfc5rddf65}
 
 Our aim in this section is to lift the set of all Sheffer sequences over $\mathbb F$ to a set of Sheffer sequences over $V$. To this end, we assume that $V$ is an associative algebra over $\mathbb F$. We denote by $v\diamond w$ the product of $v$ and $w$ from $V$ in this algebra. Note that the $\diamond$ product is not assumed to be commutative.

For each $n\geq 2$, we define a linear map $\mathbb D_n\in\mathcal L(V^{\otimes n}, V)$ by 
$$\mathbb D_n(v_1\otimes v_2\otimes \cdots\otimes v_n):=v_1\diamond v_2\diamond \cdots \diamond v_n,\quad v_1,\dots,v_n\in V.$$
By taking the restriction of $\mathbb D_n$ to $V^{\odot n}$, we get a map $\mathbb D_n\in\mathcal L(V^{\odot n}, V)$. 
We have $$\mathbb D_n(v_1\odot v_2\odot\cdots\odot v_n)=\frac 1{n!}\sum_{\pi\in S_n}v_{\pi_{(1)}}\diamond v_{\pi_{(2)}}\diamond \dotsm \diamond v_{\pi_{(n)}},\quad v_1,\dots,v_n\in V.$$
Note that, if the $\diamond$ product is commutative, then
$$\mathbb D_n(v_1\odot v_2\odot\cdots\odot v_n)=v_1\diamond v_2\diamond \cdots\diamond v_n,\quad v_1,\dots,v_n\in V.$$
In any case, $\mathbb D_nv^{\otimes n}=v^{\diamond n}$ for $v\in V$. 
Let also $\mathbb D_1$ denote the identity operator in $V$. 
We also fix a linear functional $\Upsilon\in V^*$.

Consider formal power series $a(t)=\sum_{k=1}^\infty a_kt^k$, $b(t)=\sum_{k=1}^\infty b_kt^k$ from $ \mathcal F(\mathbb F,\mathbb F)$ that satisfy $a(0)=b(0)=0$ and $b'(0)=b_1\ne0$. Let $(s_n(x))_{n=0}^\infty$ denote the Sheffer sequence  over~$\mathbb F$ that has the generating function 
$\sum_{n=0}^\infty \frac{t^n}{n!}s_n(x)=\exp\big[xb(t)+a(t)\big]$.

For each $n\in \mathbb N$, we define $\alpha_n\in (V^{\odot n})^*$ and $B_n\in \mathcal L(V^{\odot n},V)$  by
$\alpha_n:=a_n \Upsilon \mathbb D_n$ and $B_n:=b_n\mathbb D_n$. Note that $B_1=b_1\mathbf 1$ with $b_1\ne0$, hence the operator $B_1$ is invertible.  
Consider the   formal tensor power series 
\begin{align}
\alpha(v)&:=\sum_{k=1}^\infty \langle \alpha_k,v^{\otimes k}\rangle
=\sum_{k=1}^\infty a_k \langle \Upsilon, v^{\diamond k}\rangle=\bigg\langle\Upsilon,\sum_{k=1}^\infty a_k v^{\diamond k}\bigg\rangle,\notag\\
B(v)&:=\sum_{k=1}^\infty B_kv^{\otimes k}=\sum_{k=1}^\infty b_k v^{\diamond k}.\label{vxdxdzxts}
\end{align}

We will say that a {\it Sheffer sequence $(S_n(\omega))_{n=0}^\infty$ over $V$ is a lifting of the Sheffer sequence $(s_n(x))_{n=0}^\infty$} if the generating function of $(S_n(\omega))_{n=0}^\infty$ is given by formula  \eqref{fgh784gh74} in which $\alpha(v)$ and $B(v)$ are defined by \eqref{vxdxdzxts}. Below, we assume that $(S_n(\omega))_{n=0}^\infty$ is such a lifted Sheffer sequence. 

Let $l(t)=\sum_{k=1}^\infty l_kt^k$ be the compositional inverse of the formal power series $b(t)$. Recall that $L(v)=\sum_{k=1}^\infty L_kv^{\otimes k}=B^{\langle-1\rangle}(v)$. It is easy to check that, for each $k\in\mathbb N$, $L_k=l_k\mathbb D_k$, i.e.,
\begin{equation}\label{xresare5w6u}
L(v)=\sum_{k=1}^\infty l_kv^{\diamond k}. 
\end{equation}

For $k\in\mathbb N$, we define $\nabla^{k}\in\mathcal L\big(\mathcal P(V^*),\mathcal P(V^*,V)\big)$ by $(\nabla^{k}p)(\omega):=\mathbb D_k(D^{k}p)(\omega)$ for $p\in\mathcal P(V^*)$. Thus, 
\begin{equation}\label{xese4a43q4a}
\big(\nabla^{k}\langle \cdot^{\otimes n}, v^{\otimes n}\rangle\big)(\omega)=(n)_k\,v^{\diamond k}\langle \omega^{\otimes(n-k)},v^{\otimes(n-k)}\rangle.
\end{equation}
For a formal power series $d(t)=\sum_{k=1}^\infty d_kt^k\in\mathcal F(\mathbb F,\mathbb F)$ with $d(0)=0$, we define a linear operator $d(\nabla)\in\mathcal L\big(\mathcal P(V^*),\mathcal P(V^*,V)\big)$ by $d(\nabla):=\sum_{k=1}^\infty d_k\nabla^k$.

By Theorem \ref{ctrst5u} (S3), Theorem \ref{45678765g} (B3), and formulas \eqref{xresare5w6u}, \eqref{xese4a43q4a}, we conclude that the lowering differential  $Q$ for the Sheffer sequence $(S_n(\omega))_{n=0}^\infty$ has  the representation $Q=l(\nabla)$.

Denote $\tilde a_k:=k!\,a_k$, $\tilde b_k:=k!\,b_k$. Then, by formula \eqref{dxtststs},  
$$\langle S_n(\omega),v^{\otimes n}\rangle=\sum_{\pi\in\mathcal P(n)}\prod_{\mathcal A\in\pi}\big(\tilde b_{|\mathcal A|}\langle\omega,  v^{\diamond|\mathcal A|}\rangle+\tilde a_{|\mathcal A|}\langle \Upsilon,v^{\diamond|\mathcal A|}\rangle\big).$$

\begin{proposition}\label{dfgh37892g2} Let $v$ be an idempotent element of the algebra $V$, i.e., $v^{\diamond2}=v$. Let $(\hat s_n)_{n=0}^\infty$ be the Sheffer sequence over $\mathbb F$ that has the generating function
$\sum_{n=0}^\infty \frac{t^n}{n!}\,\hat s_n(x)=\exp\big[xb(t)+\langle \Upsilon, v\rangle a(t)\big]$.
Then, for all $\omega\in V^*$, $\langle S_n(\omega),v^{\otimes n}\rangle=\hat s_n(\langle \omega,v\rangle)$. 
In particular, if $\langle \Upsilon,v\rangle=1$, then 
$\langle S_n(\omega),v^{\otimes n}\rangle=s_n(\langle\omega,v\rangle)$.
\end{proposition}

\begin{proof} Since $v$ is idempotent, $v^{\diamond k}=v$ for all $k\geq2$. Hence, for $t\in\mathbb F$, formula \eqref{vxdxdzxts} implies $\alpha(tv)=\langle \Upsilon,v\rangle a(t)$ and $B(tv)=vb(t)$. Therefore,
\begin{align*}\sum_{n=0}^\infty\frac{t^n}{n!}\langle S_n(\omega),v^{\otimes n}\rangle
&=\sum_{n=0}^\infty\frac1{n!}\langle S_n(\omega),(tv)^{\otimes n)}\rangle=\exp\big[\langle\omega,v\rangle b(t)+\langle \Upsilon,v\rangle a(t)\big]\\
&=\sum_{n=0}^\infty\frac{t^n}{n!} \hat s_n(\langle\omega,v\rangle),\end{align*}
which implies the proposition. 
\end{proof}

\begin{corollary} \label{783eyeh32} Let $e_1,\dots,e_K\in V$ be such that $e_i\diamond e_j=\delta_{i,j}\,e_i$ and $\langle\Upsilon,e_i\rangle=1$ for all $i,j\in\{1,\dots,K\}$. Here $\delta_{i,j}$ is the Kronecker delta.  Then, for any $l_1,\dots,l_K\in \mathbb N_0$, $l_1+\cdots+l_K=n$, we have
\begin{equation}\label{567yhg7}\langle S_n(\omega),e_1^{\odot l_1}\odot\cdots\odot e_K^{\odot l_K}\rangle=s_{l_1}(\langle\omega,e_1\rangle)\cdots s_{l_K}(\langle\omega,e_K\rangle),\end{equation}
and for any $c_1,\dots, c_K\in \mathbb F$, we have
\begin{align}
&\big\langle S_n(\omega),(c_1e_1+\cdots+c_Ke_K)^{\otimes n}\big\rangle\notag\\
&\quad=\sum_{{l_1,\dots,l_K\in \mathbb N_0}\atop{l_1+\cdots+l_K=n}}{n\choose{ l_1\cdots l_K}}c_1^{l_1}\cdots c_K^{l_K}\, s_{l_1}(\langle\omega,e_1\rangle)\cdots s_{l_K}(\langle\omega,e_K\rangle).\label{vcfdstrjsj6u}\end{align}
\end{corollary}

\begin{proof} Formula \eqref{567yhg7} follows immediately  from Corollary~\ref{stesw5twu5} and Proposition \ref{dfgh37892g2}. Formula~ \eqref{567yhg7}  and the multinomial theorem imply formula~\eqref{vcfdstrjsj6u}. 
\end{proof}

We define the following formal power series:
\begin{align*}
\phi(t):=&b'(l(t))=\frac{1}{l'(t)}=b_1+\sum_{k=1}^\infty \phi_kt^k,\quad
\gamma(t):=a'(l(t))=a_1+\sum_{k=1}^\infty \gamma_kt^k,\\
\psi(t):=&l'(b(t))=\frac1{b'(t)}=l_1+\sum_{k=1}^\infty \psi_kt^k,\quad \delta(t):=-a'(t)\psi(t)=\sum_{k=0}^\infty\delta_kt^k.
\end{align*} 
We also define $\widehat\phi(t):=\phi(t)-b_1=\sum_{k=1}^\infty \phi_kt^k$ and $\widehat\gamma(t):=\gamma(t)-a_1=\sum_{k=1}^\infty \gamma_kt^k$.

\begin{proposition}\label{bdtyjfdese} 
{\rm (i)} For each $\zeta\in V$,  the  raising operator $R(\zeta)$  for the lifted Sheffer sequence $(S_n(\omega))_{n=0}^\infty$ acts as follows:
\begin{align}
&\big(R(\zeta)p\big)(\omega)=
\big(b_1\langle\omega,\zeta\rangle +a_1\langle\Upsilon,\zeta\rangle\big)p(\omega)\notag\\
&\quad+\langle\omega, \zeta\diamond (\widehat\phi(\nabla)p)(\omega)\big\rangle+ \big\langle\Upsilon,\zeta\diamond(\widehat\gamma(\nabla)p)(\omega)\big\rangle, \quad p\in\mathcal P(V^*),\ \omega\in V^*.\label{fxdtsxtsts}
\end{align}

{\rm (ii)} For each $\zeta, v\in V$ and $n\in\mathbb N_0$,
 \begin{align}
 &\langle\omega,\zeta\rangle\langle S_n(\omega),v^{\otimes n}\rangle=\langle S_{n+1}(\omega),v^{\otimes n}\odot\zeta\rangle\notag\\
&
 +\sum_{m=1}^{n}\Big\langle S_m(\omega), \psi_{n-m+1}(n)_{n-m+1} v^{\otimes(m-1)}\odot (\zeta\diamond v^{\diamond(n-m+1)})\notag\\
 &\qquad\quad+\delta_{n-m}(n)_{n-m}v^{\otimes m}\langle \Upsilon,\zeta\diamond v^{\diamond(n-m)}\rangle\Big\rangle+\delta_n\, n!\,\langle\Upsilon,\zeta\diamond v^{\diamond n}\rangle.\label{cxeyft7f76ui8uytr}
 \end{align}
\end{proposition}

\begin{proof} 
(i) By \eqref{vcryds6uei}, formula \eqref{fxdtsxtsts} obviously holds when $p(\omega)$ is a constant. 
Let $m\in\mathbb N$, $\zeta,v\in V$ and $p(\omega)=\langle\omega^{\otimes m},v^{\otimes m}\rangle$. Then, by formula~\eqref{vcryds6uei}, we obtain
\begin{align*}
&(R(\zeta)p)(\omega)-\big(b_1\langle\omega,\zeta\rangle +a_1\langle\Upsilon,\zeta\rangle\big)p(\omega)\\
&=\sum_{k=1}^m\sum_{n=1}^k(n+1) \sum_{i_1,\dots,i_n \geq 1 \atop i_1+\dots+i_n=k} (m)_k\,
l_{i_1}\dotsm l_{i_n}\\
&\quad\times
\big\langle \omega b_{n+1}\mathbb D_{n+1}+a_{n+1}\Upsilon\mathbb D_{n+1}, \zeta\odot v^{\diamond l_{i_1}}\odot\dotsm\odot v^{\diamond l_{i_n}}    \big\rangle\langle\omega^{\otimes(m-k)},v^{\otimes(m-k)}\rangle\\
& =\sum_{k=1}^m\sum_{n=1}^k(n+1) \sum_{i_1,\dots,i_n \geq 1 \atop i_1+\dots+i_n=k} (m)_k\,
l_{i_1}\dotsm l_{i_n}
\big\langle b_{n+1}\omega +a_{n+1}\Upsilon, \zeta\diamond  v^{\diamond k}    \big\rangle\langle\omega^{\otimes(m-k)},v^{\otimes(m-k)}\rangle\\
& =\sum_{k=1}^m\sum_{n=1}^k(n+1) \sum_{i_1,\dots,i_n \geq 1 \atop i_1+\dots+i_n=k} 
l_{i_1}\dotsm l_{i_n}
\big( b_{n+1}\langle \omega,\zeta\diamond(\nabla^kp)(\omega)\rangle+a_{n+1}\langle\Upsilon,\zeta\diamond(\nabla^kp)(\omega)\rangle\big)\\
&=\langle\omega, \zeta\diamond (\widehat\phi(\nabla)p)(\omega)\big\rangle+ \big\langle\Upsilon,\zeta\diamond(\widehat\gamma(\nabla)p)(\omega)\big\rangle.
\end{align*}

(ii) We easily check that, for any $\zeta,v\in V$, we have 
\begin{equation}\Psi_{\zeta,k}v^{\otimes k}=\psi_k\, v^{\diamond k}\diamond \zeta,\quad
\langle \Delta_{\zeta,k},v^{\otimes k}\rangle=\delta_k\langle\Upsilon, v^{\diamond k}\diamond \zeta\rangle. 
\label{cfxfdzzre}
\end{equation}
Formulas \eqref{cxesa5yw7ypui8uytr} and \eqref{cfxfdzzre} imply \eqref{cxeyft7f76ui8uytr}. 
\end{proof}

\subsection{Examples}

We will now discuss several examples of the constructions in  Section \ref{87ygfc5rddf65}. 

\begin{example}\label{vctstszzz}
Assume that a vector space $V$ has a basis $(e_k)_{k\geq1}$, i.e., each $v\in V$ can be uniquely written as $v=\sum_{k=1}^K c_ke_k$  for some $K\in\mathbb N$. Then formula 
$e_k\diamond e_l:=\delta_{k,l}e_k$ ($k,l\ge1$) uniquely determines an associative product in $V$.  It is natural to call this product in $V$ the {\it Hadamard product}.
Also the formula $\langle \Upsilon,e_k\rangle:=1$ for all $k\ge1$ uniquely determines a linear functional $\Upsilon\in V^*$. Then formula~\eqref{vcfdstrjsj6u} provides a description of each Sheffer polynomial $\langle S_n(\omega),v^{\otimes n}\rangle$  ($v\in V$) through the Sheffer sequence $(s_n)_{n=0}^\infty$.

\end{example}

\begin{example} This is a special case of Example \ref{vctstszzz}. Let $V=\mathbb F^K$ where $K\in\mathbb N$, $K\ge2$. For $k=1,\dots, K$, we define $e_k:=(0,\dots,0,1,0,\dots)$,  where $1$ is at the $k$th place. Thus, for each $x=(x_1,\dots,x_K)\in \mathbb F^k$, we have $x=\sum_{k=1}^K x_ke_k$. Then $V^*$ can be naturally identified with $\mathbb F^K$: for each $\omega=(\omega_1,\dots,\omega_K)\in V^*=\mathbb F^K$, we have
$\langle \omega,x\rangle=\sum_{k=1}^K\omega_kx_k$.
Thus, by Corollary~\ref{783eyeh32}, for $\omega=(\omega_1,\dots,\omega_K)\in \mathbb F^K$ and any $l_1,\dots,l_K\in \mathbb N_0$, $l_1+\cdots+l_K=n$, we get 
\begin{equation*}\langle S_n(\omega),e_1^{\otimes l_1}\odot\cdots\odot e_K^{\otimes l_K}\rangle=s_{l_1}(\omega_1)\cdots s_{l_K}(\omega_K).\end{equation*}
\end{example}

\begin{example}\label{987yt34rtyhj98ufd} This is again a special case of Example \ref{vctstszzz}. Let $V=\mathbb F_{\mathrm{fin}}^\infty$,
 be the vector space of all infinite sequences $x=(x_k)_{k=1}^\infty$ ($x_k\in \mathbb F$  for all $ k\in \mathbb N$) such that there exists $K\in \mathbb N$ (depending on $x$) for which  $x_k=0$   for all $k\geq K+1$. Then $V^*=\mathbb F^\infty$, the vector space of all infinite sequences $\omega=(\omega_k)_{k=1}^\infty$ with $\omega_k\in \mathbb F$. The dual pairing between $\omega=(\omega_k)_{k=1}^\infty\in\mathbb F^\infty$ and $x=(x_k)_{k=1}^\infty\in \mathbb F_{\mathrm{fin}}^\infty$ is given by
$\langle\omega,x\rangle:=\sum_{k=1}^\infty \omega_kx_k$. (Note that the series in this formula 
has only a finite number of non-zero terms.) For $k\in\mathbb N$, denote $e_k:=(0,\dots,0,1,0,\dots)$, where $1$ is at the $k$th place. Then $(e_k)_{k=1}^\infty$ forms a basis for $\mathbb F_{\mathrm{fin}}^\infty$. 
For $x=(x_k)_{k=1}^\infty$ and $y=(y_k)_{k=1}^\infty$ from $\mathbb F_{\mathrm{fin}}^\infty$, 
the Hadamard product of $x$ and $y$ has the form  $x\diamond y=(x_ky_k)_{k=1}^\infty$. We also have $\langle \Upsilon,x\rangle=\sum_{k=1}^\infty x_k$ (again the series has only a finite number of non-zero terms). Thus, by Corollary~\ref{783eyeh32}, we get, for each $\omega=(\omega_k)_{k=1}^\infty\in \mathbb F^\infty$, $K\in\mathbb N$, $l_1,\dots,l_K\in \mathbb N_0,\, l_1+\cdots+l_K=n$,
\begin{equation*}\langle S_n(\omega),e_1^{\odot l_1}\odot e_2^{\odot l_2}\odot\cdots\odot e_K^{\odot l_K}\rangle=s_{l_1}(\omega_1)s_{l_2}(\omega_2)\cdots s_{l_K}(\omega_K).\end{equation*}
\end{example}

\begin{example}\label{cxzfrzrzrs} Let $V=\mathcal D(\mathbb R^d)=C_0^{\infty}(\mathbb R^d)$ be the vector space of all real-valued smooth functions on $\mathbb R^d$ with compact support. Note that $\mathcal D(\mathbb R^d)^*$ includes $\mathcal D(\mathbb R^d)'$, the topological dual of the space $\mathcal D(\mathbb R^d)$. The $\mathcal D(\mathbb R^d)$ is a commutative associative algebra for the point-wise product of functions: $(f\diamond g)(x):=f(x)g(x)$. We also define $\Upsilon\in \mathcal D(\mathbb R^d)'\subset \mathcal D(\mathbb R^d)^*$ by $\langle \Upsilon,f\rangle:=\int_{\mathbb R^d}f(x)\,dx$ for $f\in \mathcal D(\mathbb R^d)$.
As easily seen, the generating function of the lifted Sheffer sequence has the following form: for $\omega\in \mathcal D(\mathbb R^d)^*$ and $f\in \mathcal D(\mathbb R^d)$, 
\begin{equation}\label{dfghjk0987654dcv}
\sum_{n=1}^\infty \frac1{n!}\langle S_n(\omega),f^{\otimes n}\rangle=
\exp\bigg[\langle \omega,b(f)\rangle+\int_{\mathbb R}a(f(x))\,dx\bigg].\end{equation}
The reader is advised to compare this example with \cite[Section~7]{FKLO}.
\end{example}

\begin{example}\label{fxdtsts6iu} Similarly to Example~\ref{cxzfrzrzrs}, we may consider $V=\mathcal S(\mathbb R^d)$, the Schwartz space of real-valued smooth rapidly decreasing functions on $\mathbb R^d$. Recall that $f\in \mathcal S(\mathbb R^d)$ if and only if the function $f:\mathbb R^d\to \mathbb R$ is smooth and any  partial derivatives of~$f$, multiplied by an arbitrary polynomial on $\mathbb R^d$ is a bounded function.  Note that $\mathcal S(\mathbb R^d)^*$ includes $\mathcal S(\mathbb R^d)'$, the topological dual of $\mathcal S(\mathbb R^d)$, the Schwartz space of tempered distributions. It is well known that $\mathcal S(\mathbb R^d)$ is a commutative associative algebra for the point-wise multiplication of functions. We also define $\Upsilon\in S(\mathbb R)'\subset S(\mathbb R)^*$ by 
$\langle \Upsilon,f\rangle:=\int_{\mathbb R}f(x)\,dx$. The generating function of a lifted Sheffer sequence is again of the form \eqref{dfghjk0987654dcv}. \end{example}

We note that both Examples~\ref{cxzfrzrzrs} and \ref{fxdtsts6iu} can be extended to the case of complex-valued functions from the corresponding spaces $\mathcal D(\mathbb R^d,\mathbb C)$ and  
$\mathcal S(\mathbb R^d,\mathbb C)$. 

\begin{example}Let $(X,\mathcal F)$ be a measurable  space, i.e., $X$ is a set and $\mathcal F$ is a $\sigma$-algebra on~$X$. Let $B(X,\mathbb F)$ denote the vector space of all $\mathbb F$-valued measurable bounded functions on~$X$. Then $B(X,\mathbb F)$ is a commutative associative algebra for the point-wise product of functions. For a finite measure $\mu$ on $(X,\mathcal F)$, we define $\Upsilon\in B(X)^*$ by $\langle \Upsilon,f\rangle:=\int_Xf\,d\mu$ ($f\in B(X,\mathbb F)$).
Then the generating function of the lifted Sheffer sequence has the following form, for $\omega\in B(X,\mathbb C)^*$ and $f\in B(X,\mathbb F)$, 
\begin{equation*}\sum_{n=0}^\infty\frac1{n!}\langle S_n(\omega),f^{\otimes n}\rangle=\exp\bigg[\langle\omega,b(f)\rangle+\int_Xa(f(x))\,d\mu (x)\bigg].\end{equation*}
For each $\Lambda\in\mathcal F$ with $\mu(\Lambda)>0$, the indicator function of $\Lambda$, denoted by $\chi_\Lambda$, is an idempotent element of this algebra: $\chi_\Lambda^2=\chi_\Lambda$. Furthermore, for any $k$ mutually disjoint sets $\Lambda_1,\dots,\Lambda_K$  in $\mathcal F$ with $\mu(\Lambda_i)>0$ for all $i=1,\dots,K$, we have 
$\chi_{\Lambda_i}\chi_{\Lambda_j}=\delta_{ij}\chi_{\Delta_i}$.
Hence, we can apply to this case a proper modification of Corollary~\ref{783eyeh32} that takes into account that $\langle \Upsilon,\chi_{\Lambda_i}\rangle=\mu(\Lambda_i)$ is not necessarily equal to $1$, compare with Proposition~\ref{dfgh37892g2}.
\end{example}

\begin{example}Let $n\geq2$ and let $V=\mathcal M_n(\mathbb F)$ be the vector space of all $n\times n$ matrices with entries from
  $\mathbb F$.  The $\mathcal M_n(\mathbb F)$ is an algebra under addition of matrices, multiplication by a constant and multiplication of matrices. The dual space $\mathcal M_n(\mathbb F)^*$  can be identified with $\mathcal M_n(\mathbb F)$ if the dual pairing between $\omega=[\omega_{ij}]_{i,j=1,\dots,n}\in \mathcal M_n(\mathbb F)^*=\mathcal M_n(\mathbb F)$ and $M=[m_{ij}]_{i,j=1,\dots,n}\in\mathcal M_n(\mathbb F)$ is given by 
$\langle \omega,M\rangle:=\sum_{i,j=1}^n \omega_{ij}m_{ij}$. 
The idempotent elements of the algebra $\mathcal M_n(\mathbb F)$ are exactly the idempotent matrices from $\mathcal M_n(\mathbb F)$. A natural candidate for the functional $\Upsilon\in \mathcal M_n(\mathbb F)^*$ is the trace of the matrix. Thus, the generating function of the lifted Sheffer sequence has the following form: for $\omega,M\in \mathcal M_n(\mathbb F)$,
\begin{equation}\label{fxzrasrysu5}
\sum_{n=0}^\infty\frac1{n!}\langle S_n(\omega),M^{\otimes n}\rangle=\exp\big[\langle\omega,b(M)\rangle+\operatorname{Tr}a(M)\big],
\end{equation} 
where $a(M):=\sum_{k=1}^\infty a_kM^k$, $b(m):=\sum_{k=1}^\infty b_kM^k$. 
\end{example}

\begin{example}\label{hish9wss9ss9}
For each $n\geq2$, consider the embedding $I_n:\mathcal M_n(\mathbb F)\to \mathcal M_{n+1}(\mathbb F)$ given by $I_nM=\widetilde M$, where $M=[m_{ij}]_{i,j=1,\dots,n}$, $\widetilde M=[\tilde m_{ij}]_{i,j=1,\dots,n+1}$, with $\tilde m_{ij}=m_{ij}$ if $\max\{i,j\}\leq n$ and $\tilde m_{ij}=0$ otherwise. In view of this embedding, we may treat $\mathcal M_n(\mathbb F)$ as a vector subspace of $\mathcal M_{n+1}(\mathbb F)$. Obviously, $\mathcal M_n(\mathbb F)$ is, in fact, a subalgebra of $\mathcal M_{n+1}(\mathbb F)$. Consider the vector space  $\mathcal M_{\mathrm{fin}}(\mathbb F):=\bigcup_{n=1}^\infty \mathcal M_n(\mathbb F)$. Then $\mathcal M_{\mathrm{fin}}(\mathbb F)$ is an  algebra for addition of matrices, multiplication by a constant, and multiplication of matrices. The dual space $\mathcal M_{\mathrm{fin}}(\mathbb F)^*$ can be identified with $\mathcal M_\infty(\mathbb F)$, the space of all infinite matrices $\omega=[\omega_{ij}]_{i,j\in\mathbb N}$ and the dual pairing between $\omega=[\omega_{ij}]_{i,j\in\mathbb N}\in \mathcal M_\infty(\mathbb F)$ and $M=[m_{ij}]_{i,j\in\mathbb N}\in\mathcal M_{\mathrm{fin}}(\mathbb F)$ is given by $\langle\omega,M\rangle:=\sum_{i,j=1}^\infty \omega_{ij}m_{ij}$. (The sum on the right-hand side of this formula is, in fact, finite.)
Idempotent elements of the algebra  $\mathcal M_{\mathrm{fin}}(\mathbb F)$ are exactly the idempotent matrices from $\mathcal M_n(\mathbb F)$ with $n\geq 2$. Again a natural candidate for the functional $\Upsilon\in \mathcal M_\infty(\mathbb F)$ is  the trace of the matrix. The generating function of the lifted Sheffer sequence, for $\omega\in \mathcal M_\infty(\mathbb F)$ and $M\in \mathcal M_{\mathrm{fin}}(\mathbb F)$, is given by formula~\eqref{fxzrasrysu5}.  \end{example}

\begin{example} Let $V$ be a Banach $*$-algebra. Then the dual space $V^*$ contains the topological dual $V'$, which is a Banach space. A natural candidate for the functional $\Upsilon$ is a state on $V$, i.e., $\Upsilon\in V'$ satisfies $\Upsilon(v^*\diamond v)\ge0$ for all $v\in V$. In the case where $V$ contains the identity element $1$, one may use a unital state  $\Upsilon$, which satisfies $\Upsilon(1)=1$. 
\end{example}

\begin{center}
Acknowledgements 
\end{center}
AA is grateful to Taif University for the financial support of his PhD studies at Swansea University. 
The authors would like to thank Maria Jo\~{a}o Oliveira for numerous useful discussions and her feedback on the preliminary version of the paper.

\end{document}